\documentclass[a4paper, 10pt]{amsart}
\usepackage[utf8]{inputenc}
\usepackage[english]{babel}
\usepackage[T1]{fontenc}
\usepackage{amsmath, amssymb, amsfonts}
\usepackage[all]{xy}
\usepackage{enumerate}
\usepackage{graphicx}
\usepackage[hmargin=1.5cm, vmargin=2.5cm]{geometry}
\usepackage{rotate}
\usepackage{MnSymbol}
\usepackage{comment}
\usepackage{tikz}
\usepackage{hyperref}

\title[Partition quantum groups]{On the partition approach to Schur-Weyl duality and free quantum groups}
\author{Amaury Freslon}
\keywords{Quantum groups, Schur-Weyl duality}
\subjclass[2010]{05E10, 05A18}
\address{Saarland University, Fachbereich Mathematik, Postfach 151159, 66041 Saarbr\"ucken, Germany}
\email{freslon@math.uni-sb.de}
\date{}

\theoremstyle{plain}
\newtheorem{thm}{Theorem}[subsection]
\newtheorem{prop}[thm]{Proposition}
\newtheorem{cor}[thm]{Corollary}
\newtheorem{lem}[thm]{Lemma}
\newtheorem{conj}[thm]{Conjecture}

\theoremstyle{definition}
\newtheorem{de}[thm]{Definition}
\newtheorem{ex}[thm]{Example}

\theoremstyle{remark}
\newtheorem{rem}[thm]{Remark}

\DeclareMathOperator{\Aut}{Aut}

\DeclareMathOperator{\Id}{Id}

\DeclareMathOperator{\Irr}{Irr}

\DeclareMathOperator{\Mor}{Hom}

\DeclareMathOperator{\Proj}{Proj_{\CC}}

\DeclareMathOperator{\rl}{rl}
\DeclareMathOperator{\Span}{Span}
\DeclareMathOperator{\Sym}{Sym_{\CC}}

\newcommand{\A}{\mathcal{A}}
\newcommand{\B}{\mathcal{L}}
\newcommand{\C}{\mathbb{C}}
\newcommand{\CC}{\mathcal{C}}

\newcommand{\D}{\Delta}

\newcommand{\G}{\mathbb{G}}
\newcommand{\Gr}{\mathcal{G}}
\newcommand{\HH}{\mathcal{H}}

\newcommand{\N}{\mathbb{N}}

\newcommand{\OO}{\mathcal{O}}
\newcommand{\Pc}{P^{\A}}

\newcommand{\R}{\mathbb{R}}

\newcommand{\UU}{\mathcal{U}}
\newcommand{\Z}{\mathbb{Z}}
\newcommand{\ZZ}{\mathcal{Z}}

\newcommand{\co}{\overline}
\newcommand{\h}{\widehat}
\newcommand{\ii}{\imath}

\newcommand{\crosspart}{
\mathrel{\offinterlineskip
\hbox{$/$}\hskip -.95ex\hbox{$\backslash$}}}

\usepackage{calc}
\newcounter{PartitionDepth}
\newcounter{PartitionLength}

\begin{document}

\begin{abstract}
We give a general definition of classical and quantum groups whose representation theory is "determined by partitions" and study their structure. This encompasses many examples of classical groups for which Schur-Weyl duality is described with diagram algebras as well as generalizations of P. Deligne's interpolated categories of representations. Our setting is inspired by many previous works on \emph{easy quantum groups} and appears to be well-suited to the study of free fusion semirings. We classify free fusion semirings and prove that they can always be realized through our construction, thus solving several open questions. This suggests a general decomposition result for free quantum groups which in turn gives information on the compact groups whose Schur-Weyl duality is implemented by partitions. The paper also contains an appendix by A. Chirvasitu proving simplicity results for the reduced C*-algebras of some free quantum groups.
\end{abstract}

\maketitle

\section{Introduction}

Partitions of finite sets are \emph{a priori} very simple set-theoretic objects. However, their combinatorics appear to be quite rich and plays a role in many different areas of mathematics. For example, R. Brauer introduced in \cite{brauer1937algebras} algebras generated by partitions in pairs to study the invariants of tensor representations of the orthogonal and symplectic groups. His ideas were later developed by several authors to describe \emph{Schur-Weyl duality} for other classes of groups like complex reflection groups or wreath products (see Section \ref{sec:examples}). More precisely, several classes of algebras generated by partitions where introduced and it was proved that they are isomorphic to centralizer algebras for certain tensor representations of the involved groups. These ideas were reinterpreted and extended in a more categorical setting by F. Knop and M. Mori using the idea of \emph{interpolated categories of representations} introduced by P. Deligne for symmetric groups in \cite{deligne2007categorie}. In a different context, it was discovered by R. Speicher that passing from the combinatorics of all partitions to that of the subclass of \emph{noncrossing partitions} translates in a probabilistic setting into passing from classical (tensor) independence to free independence (see for instance the book \cite{nica2006lectures}). 

These probabilistic ideas motivated the introduction by T. Banica and R. Speicher in \cite{banica2009liberation} of a class of compact quantum groups called \emph{easy quantum groups}, whose representation theory is ruled by the combinatorics of (noncrossing) partitions. This gave a new point of view on previous results of T. Banica linking the representation theory of orthogonal quantum groups with Temperly-Lieb algebras and suggested strong links with the work of R. Brauer. The construction also gave some insight into the notion of "freeness" for quantum groups, which is crucial to study their geometric properties (see the Appendix). Moreover, the class of easy quantum groups contains some classical groups for which, by definition, Schur-Weyl duality can be described by partition algebras (for instance the orthogonal groups $O_{N}$ or the symmetric groups $S_{N}$).

These ideas may be extended in several ways. The present paper is motivated by the two following facts : first, few examples of groups whose Schur-Weyl duality is described by partitions could be included in this setting and secondly, few examples of "free" quantum groups could be constructed. Looking at the classical case (for example the work of P. Glockner and W. von Waldenfels \cite{glockner1989relations} on unitary groups) it seems natural to extend the setting by colouring the points of the partitions. It can then be hoped that the various approaches will be unified and that new phenomena will occur. This is precisely what we will endeavour in this work. Our aim is therefore threefold :
\begin{enumerate}
\item Give a general setting together with a comprehensive description of all the known results concerning representation theory of quantum groups "determined by partitions".
\item Try to encompass as much as possible works on classical Schur-Weyl duality in this setting, in order to unify these works, as well as categorical approaches in the spirit of P. Deligne.
\item See how far noncrossing partitions are linked with the notion of "freeness" of a quantum group.
\end{enumerate}

As will appear in Section \ref{sec:general}, point $(1)$ is fulfilled in a quite satisfying way. For point $(2)$ however, we will see in Section \ref{sec:examples} that  some general wreath products cannot be directly described by partition quantum groups but need some averaging of the partitions. This is close to some examples of "super-easy" quantum groups introduced by T. Banica and A. Skalski in \cite{banica2011two}. We will suggest in Subsection \ref{subsec:wreath} a way of unifying these examples and our framework. The main achievement of this paper concerns point $(3)$. In fact, we will prove that any free quantum group has the same representation theory as a partition quantum group. This will first be proved in Theorem \ref{thm:fullclassification} by constructing a suitable category of partitions. Then, using a classification of all free fusion semirings, we will be able to deduce in Theorem \ref{thm:freeproduct} a very simple form for the partition quantum group associated to a given free fusion semiring. We will also conjecture in Subsection \ref{sec:nonunimodular} a classification of \emph{all} free quantum groups.

Let us now outline the contents of the paper. Section \ref{sec:partitions} gathers necessary background on partitions. The constructions are quite standard but we give them for the sake of completeness. We define partition quantum groups in Section \ref{sec:general} through a general Tannaka-Krein duality argument and give general results on their representation theory. These results are simple adaptations of the uncoloured case, which may make this section look like a survey. We think however that it is necessary to include them in order to keep this work as self-contained as possible. We turn to examples in Section \ref{sec:examples}. On the one hand we recover all the known quantum examples and on the other hand we also explain how many classical examples (in particular wreath products of abelian compact groups) of combinatorial Schur-Weyl duality fit into our setting. We end with some comments on other cases, where the description requires some averaging procedure on partitions.

Section \ref{sec:fusion} is the most technical part of the paper. We first explain how previous results on fusion semirings of noncrossing quantum groups extend to the partition setting, in particular characterizing when it is free in Theorem \ref{thm:freefusion}. We then prove in Theorem \ref{thm:fullclassification} that any free fusion semiring can be realized from a partition quantum group, which shows that our setting is optimal in this sense. We then endeavour to classify these free fusion semirings, culminating in a classification of all compact quantum groups with free fusion semiring up to $R^{+}$-deformation in Theorem \ref{thm:freeproduct}.

Eventually, we investigate in Section \ref{sec:tensor} possible extensions of our work in two directions. The first extension involves \emph{monoidal equivalence}. Assuming a technical conjecture, we classify all free quantum groups in Corollary \ref{cor:completeclassification}. An interesting consequence of this is that through an abelianization procedure, we also get a classification of all the \emph{classical} groups whose Schur-Weyl duality can be described by "block-stable" partitions. The second extension is more algebraic and linked to the subalgebras of partition algebras appearing at the end of Section \ref{sec:examples}. Following some recent work of F. Lemeux and P. Tarrago, we define a notion of "decorated" partition which contains coloured partitions. A general theory for these seems more complicated to develop, but could fill the gap between our setting and works on general wreath products and quantum isometry groups.

In the Appendix, A. Chirvasitu proves that most free quantum group have simple reduced C*-algebra together with some uniqueness property for the Haar state. He also proves a similar result for free products of arbitrary compact quantum groups.

\subsection*{Acknowledgments}

The author was supported by the ERC advanced grant "Noncommutative distributions in free probability". He moreover wishes to thank A. Skalski for fruitful discussions on topics related to partition quantum groups.

\section{Preliminaries on partitions}\label{sec:partitions}

The fundamental idea of this work is to use the combinatorics of partitions to implement centralizer algebras of groups and quantum groups, which in turn completely characterize the (quantum) group through Tannaka-Krein duality. There is a standard way to do this which was developed in \cite{banica2009liberation}. In this section, we introduce the basic material from \cite{banica2009liberation} in order to generalize it in Section \ref{sec:general} to coloured partitions. The motivation for these definitions will appear in the definition of partition quantum groups in Theorem \ref{thm:partitionqg}. For a more detailed introduction, we refer the reader to \cite{banica2009liberation} or other works on easy quantum groups.

\subsection{Partitions and crossings}

A \emph{partition} is given by two integers $k$ and $l$ and a partition $p$ of the set $\{1, \dots, k+l\}$. It is very useful to represent such partitions as diagrams, in particular for computational purposes. A diagram consists in an upper row of $k$ points, a lower row of $l$ points and some strings connecting these points if and only if they belong to the same set of the partition. Let us consider for example the partitions $p_{1} = \{\{1, 8\}, \{2, 6\}, \{3, 4\}, \{5, 7\}\}$ and $p_{2} = \{\{1, 4, 5, 6\}, \{2, 3\}\}$. Their diagram representation is :

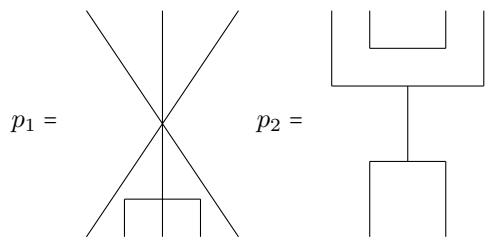
\begin{center}
\begin{tikzpicture}[scale=0.5]
\draw (0,-3) -- (0,3);
\draw (-1,-3) -- (-1,-2);
\draw (1,-3) -- (1,-2);
\draw (-1,-2) -- (1,-2);
\draw (-2,-3) -- (2,3);
\draw (2,-3) -- (-2,3);

\draw (-2.5,0) node[left]{$p_{1} = $};
\end{tikzpicture}
\begin{tikzpicture}[scale=0.5]
\draw (0,-1) -- (0,1);
\draw (-2,1) -- (2,1);
\draw (-2,1) -- (-2,3);
\draw (2,1) -- (2,3);
\draw (-1,2) -- (-1,3);
\draw (1,2) -- (1,3);
\draw (-1,2) -- (1,2);

\draw (-1,-1) -- (1,-1);
\draw (-1,-1) -- (-1,-3);
\draw (1,-1) -- (1,-3);

\draw (-2.5,0) node[left]{$p_{2} = $};
\end{tikzpicture}
\end{center}

A maximal set of points which are all connected in a partition is called a \emph{block}. We denote by $b(p)$ the number of blocks of a partition $p$ and by $t(p)$ the number of \emph{through-blocks} (also called the \emph{propagation number}), i.e. blocks containing both upper and lower points. Now that we have defined partitions, let us explain how they can be combined. We will denote by $P(k, l)$ the set of all partitions with $k$ points in the upper row and $l$ points in the lower row. The following operations will be called the \emph{category operations} :
\begin{itemize}
\item  If $p\in P(k, l)$ and $q\in P(k', l')$, then $p\otimes q\in P(k+k', l+l')$ is their \emph{horizontal concatenation}, i.e. the first $k$ of the $k+k'$ upper points are connected by $p$ to the first $l$ of the $l+l'$ lower points, whereas $q$ connects the remaining $k'$ upper points with the remaining $l'$ lower points.
\item If $p\in P(k, l)$ and $q\in P(l, m)$, then $qp\in P(k, m)$ is their \emph{vertical concatenation}, i.e. $k$ upper points are connected by $p$ to $l$ middle points and the lines are then continued by $q$ to $m$ lower points. This process may produce loops in the partition. More precisely, consider the set $L$ of elements in $\{1, \dots, l\}$ which are not connected to an upper point of $p$ nor to a lower point of $q$. The lower row of $p$ and the upper row of $q$ both induce partitions of the set $L$. The maximum (with respect to inclusion) of these two partitions is the \emph{loop partition} of $L$, its blocks are called \emph{loops} and their number is denoted by $\rl(q, p)$. To complete the operation, we remove all the loops.
\item If $p\in P(k, l)$, then $p^{*}\in P(l, k)$ is the partition obtained by reflecting $p$ with respect to an horizontal line between the two rows.
\item If $p\in P(k, l)$, then we can shift the very left upper point to the left of the lower row (or the converse) without changing the strings connecting the points in this process. This gives rise to a partition in $P(k-1, l+1)$ (or in $P(k+1, l-1)$), called a \emph{rotated version} of $p$. We can also rotate partitions on the right.
\end{itemize}

Of utmost importance in this work will be \emph{noncrossing} partitions. Informally, these are partitions such that the strings linking blocks can be drawn without crossing each other. Let us give a more formal definition. 

\begin{de}\label{de:noncrossing}
A partition $p$ is said to be \emph{crossing} if there exists four integers $k_{1} < k_{2} < k_{3} < k_{4}$ satisfying :
\begin{itemize}
\item $k_{1}$ and $k_{3}$ are in the same block.
\item $k_{2}$ and $k_{4}$ are in the same block.
\item $k_{1}$ and $k_{2}$ are not in the same block.
\end{itemize}
Otherwise, $p$ is said to be \emph{noncrossing}. Noncrossing partitions are also sometimes called \emph{planar} partitions.
\end{de}

\subsection{Categories of partitions}

As explained in \cite[Prop 3.12]{banica2009liberation}, the category operations are exactly what is needed to produce suitable C*-tensor categories. That is the reason for the introduction of the following terminology :

\begin{de}
A \emph{category of partitions} is the data of a set $\CC(k, l)$ of partitions with $k$ upper points and $l$ lower points for all integers $k$ and $l$, which is stable under the above category operations and contains the identity partition $\vert$.
\end{de}

In order to produce a (quantum) group out of a C*-tensor category using Tannaka-Krein duality, we need to make it concrete. This means that the morphism spaces shall be made of linear maps between Hilbert spaces. To do this, we need a coherent way of associating operators to partitions. This is given by the following definition \cite[Def 1.6 and 1.7]{banica2009liberation} :

\begin{de}
Let $N$ be an integer and let $(e_{1}, \dots, e_{N})$ be a basis of $\C^{N}$. For any partition $p\in P(k, l)$, we define a linear map
\begin{equation*}
T_{p}:(\C^{N})^{\otimes k} \mapsto (\C^{N})^{\otimes l}
\end{equation*}
by the following formula :
\begin{equation*}
T_{p}(e_{i_{1}} \otimes \dots \otimes e_{i_{k}}) = \sum_{j_{1}, \dots, j_{l} = 1}^{n} \delta_{p}(i, j)e_{j_{1}} \otimes \dots \otimes e_{j_{l}},
\end{equation*}
where $\delta_{p}(i, j) = 1$ if and only if all strings of the partition $p$ connect equal indices of the multi-index $i = (i_{1}, \dots, i_{k})$ in the upper row with equal indices of the multi-index $j = (j_{1}, \dots, j_{l})$ in the lower row. Otherwise, $\delta_{p}(i, j) = 0$.
\end{de}

The interplay between these maps and the category operations are given by the following rules proved in \cite[Prop 1.9]{banica2009liberation} :
\begin{itemize}
\item $T_{p}^{*} = T_{p^{*}}$.
\item $T_{p}\otimes T_{q} = T_{p\otimes q}$.
\item $T_{p}\circ T_{q} = N^{\rl(p, q)}T_{pq}$.
\end{itemize}

It is now clear that given a category of partitions $\CC$ and an integer $N$ (or a finite-dimensional Hilbert space $V$), there is an associated concrete C*-tensor category with duals (see \ref{de:tensorcategory}). We will come back to these categories in the next section after extending the setting to coloured partitions.

\section{General theory}\label{sec:general}

In this section, we will introduce and study \emph{partition quantum groups}. This first requires the introduction of \emph{colour sets} and coloured partitions. With this in hand, we will define partition quantum groups after recalling some basic notions on S.L. Woronowicz's theory of compact quantum groups. The strength of our setting is the complete description of the representation theory of these quantum groups given in the last subsection, which is a consequence of our joint work with M. Weber \cite{freslon2013representation}.

\subsection{Coloured partitions}

Intuitively, a coloured partition is a partition together with a colour attached to each point. Here by colour we simply mean an element of a fixed set. The set of colours, however, has to be endowed with an additional structure in order to yield a quantum group. The idea is that each colour corresponds to a representation of a compact quantum group and that partitions give morphisms between tensor products of the corresponding representations. In this setting, the rotation operation translates into \emph{Frobenius duality}. More precisely, if $U$, $V$ and $W$ are three representations of a compact quantum group, then there is an isomorphism
\begin{equation*}
\Mor(U, V\otimes W) \simeq \Mor(\co{V}\otimes U, W).
\end{equation*}
Since the contragredient representation $\co{V}$ of $V$ need not be equivalent to $V$, we see that the colour must be changed when rotating a point. This is the reason why we make the following definition :

\begin{de}
A \emph{colour set} is a set $\A$ endowed with an involution $x\mapsto \overline{x}$. An $\A$-coloured partition is a partition $p$ with the additional data of an element of $\A$ associated to each point of the partition.
\end{de}

Let $p$ be an $\A$-coloured partition. Reading from left to right, we can associate to the upper row of $p$ a word $w$ on $\A$ and to its lower row (again reading from left to right) a word $w'$ on $\A$. For a set of partitions $\CC$, we will denote by $\CC(w, w')$ the set of all partitions in $\CC$ such that the upper row is coloured by $w$ and the lower row is coloured by $w'$. The operations on partitions described in the previous section can be carried on to the coloured setting with the appropriate modifications. These category operations behave in the following way :
\begin{itemize}
\item If $p\in \CC(w, w')$ and $q\in \CC(z, z')$, then $p\otimes q\in \CC (w.z, w'.z')$. Here, $.$ denotes the concatenation of words.
\item If $p\in \CC(w, w')$ and $q\in \CC(w', w'')$, then $qp\in \CC(w, w'')$. Note that we can only perform this operation if the words associated to the lower row of $p$ and the upper row of $q$ match.
\item If $p\in \CC(w, w')$, then $p^{*}\in \CC(w', w)$.
\item If $w = w_{1}\dots w_{n}$, $w' = w'_{1}\dots w'_{k}$ and $p\in \CC(w, w')$, then rotating the leftmost point of the lower row of $p$ to the left of the upper row yields a partition $q\in \CC((\overline{w}'_{1}w_{1}\dots w_{n}, w'_{2}\dots w'_{k})$. In other words, rotating a point changes a colour into its conjugate. The same rotation operation can be done on the right of $p$.
\end{itemize}

Let us say that for an element $x\in \A$, the \emph{$x$-identity partition} is the partition $\vert$ coloured with $x$ on both ends. We are now ready for the definition of a category of coloured partitions.

\begin{de}
A \emph{category of $\A$-coloured partitions} $\CC$ is the data of a set of $\A$-coloured partitions $\CC(w, w')$ for all words $w$ and $w'$ on $\A$, which is stable under all the category operations and contains the $x$-identity partition for all $x\in \A$.
\end{de}

If $w$ is a word on $\A$, $\vert w\vert$ will denote its length and we will denote by $\CC(k, l)$ the set of all partitions with $k$ points on the upper row and $l$ points on the lower row. The collection of all $\A$-coloured partitions, which is clearly a category of partitions, will be denoted by $\Pc$. A coloured partition is said to be \emph{noncrossing} if the underlying uncoloured partition is noncrossing and the collection of all noncrossing partitions is denoted by $NC^{\A}$. Following the ideas of \cite{banica2009liberation}, we now have to associate linear maps to coloured partitions. This is done by simply forgetting the colours.

\begin{de}
Let $p$ be an $\A$-coloured partition and let $N$ be an integer. The map $T_{p}$ is defined to be the linear map associated to the uncoloured partition underlying $p$.
\end{de}

This means that colours only give us restrictions on the way we can compose the linear maps. If $\CC$ is a category of $\A$-coloured partitions and if $w$ and $w'$ are words on $\A$, the family of linear maps $T_{p}$ for $p\in \CC(w, w')$ need not be linearly independent. This is a source of difficulties for the study of partition groups. Dealing with the linear relations between the maps $T_{p}$ in fact amounts to the \emph{Second fundamental theorem of invariants}. In the quantum case, this linear independence problem can be ruled out for a large class of quantum groups because of the next proposition (see \cite[Lem 4.16]{freslon2013representation} for a proof).

\begin{prop}\label{prop:linearindependence}
Let $\CC$ be a category of \emph{noncrossing} $\A$-coloured partitions and let $N\geqslant 4$ be an integer. Then, for any two words $w$ and $w'$ on $\A$, the linear maps $(T_{p})_{p\in \CC(w, w')}$ are linearly independent.
\end{prop}

\subsection{Partition quantum groups}\label{subsec:quantumgroups}

As already announced, our framework will be that of \emph{compact quantum groups}, a theory which is built on \emph{C*-algebras}. We will not really make use of the analytic aspects of compact quantum groups so that the reader should not be worried about being unfamiliar with operator algebras. We will simply give a summary for completeness and in order to fix notations. We refer the reader to the book \cite{neshveyev2014compact} for details and proofs.

\begin{de}\label{DefCMQG}
A \emph{compact quantum group} is a pair $\G = (C(\G), \D)$ where $C(\G)$ is a unital C*-algebra and
\begin{equation*}
\D: C(\G)\rightarrow C(\G)\otimes C(\G)
\end{equation*}
is a unital $*$-homomorphism such that $(\D\otimes \ii)\circ\D = (\ii \otimes \D)\circ \D$ and the linear spans of $\D(C(\G))(1\otimes C(\G))$ and $\D(C(\G))(C(\G)\otimes 1)$ are dense in $C(\G)\otimes C(\G)$ (all the tensor products are spatial).
\end{de}

Here, $\ii$ denotes the identity map of the C*-algebra $C(\G)$ and $\D$ is called the \emph{coproduct}. The fundamental notion for our purpose is finite-dimensional representations.

\begin{de}
Let $\G$ be a compact quantum group. A \emph{representation} of $\G$ of dimension $n$ is a matrix
\begin{equation*}
(u_{ij})\in M_{n}(C(\G)) \simeq C(\G) \otimes M_{n}(\C)
\end{equation*}
such that $\D(u_{ij}) = \displaystyle\sum_{k} u_{ik}\otimes u_{kj}$ for every $1 \leqslant i, j \leqslant n$. The contragredient representation $\co{u}$ is defined by $\overline{u}_{ij} = u_{ij}^{*}$.
\end{de}

An \emph{intertwiner} between two representations $u$ and $v$ of dimension respectively $n$ and $m$ is a linear map $T: \C^{n} \rightarrow \C^{m}$ such that $(\ii \otimes T)u = v(\ii \otimes T)$. The set of intertwiners between $u$ and $v$ is denoted by $\Mor_{\G}(u, v)$, or simply by $\Mor(u, v)$ if there is no ambiguity. If there exists a unitary intertwiner between $u$ and $v$, then they are said to be \emph{unitarily equivalent}. A representation is said to be \emph{irreducible} if its only self-intertwiners are the scalar multiples of the identity. The \emph{tensor product} of two representations $u$ and $v$ is the representation
\begin{equation*}
u\otimes v = u_{12}v_{13}\in C(\G)\otimes M_{n}(\C)\otimes M_{m}(\C) \simeq C(\G)\otimes M_{nm}(\C),
\end{equation*}
where we used the \emph{leg-numbering} notations : for an operator $X$ acting on a twofold tensor product, $X_{ij}$ is the extension of $X$ acting on the $i$-th and $j$-th tensors of a multiple tensor product. Compact quantum groups have a tractable representation theory because of the following generalization of a classical result for compact groups.

\begin{thm}[Woronowicz]\label{thm:peterweyl}
Every unitary representation of a compact quantum group is unitarily equivalent to a direct sum of irreducible unitary representations. Moreover, any irreducible representation is finite-dimensional.
\end{thm}

We now want to state precisely the generalization of Tannaka-Krein duality proved by S.L. Woronowicz in \cite{woronowicz1988tannaka} for compact matrix quantum groups. To this end, we first introduce some notations. For a representation $u$, we set $u^{\circ} = u$ and $u^{\bullet} = \co{u}$. If $w = w_{1}\dots w_{k}$ is a word on the set $\{\circ, \bullet\}$, we set $u^{\otimes w} = u^{w_{1}}\otimes \dots \otimes u^{w_{n}}$, which is a representation acting on a Hilbert space $V^{\otimes w}$. A representation $u$ of a compact quantum group $\G$ is said to be \emph{generating} if for any irreducible representation $v$, there is a word $w$ on $\{\circ, \bullet\}$ such that $v\subset u^{\otimes w}$. By Theorem \ref{thm:peterweyl}, this implies that any finite-dimensional representation of $\G$ is a subrepresentation of a direct sum of representations of the form $u^{\otimes w}$.

\begin{de}
A \emph{compact matrix quantum group} is a pair $(\G, u)$ where $\G$ is a compact quantum group and $u$ is a finite-dimensional generating representation of $\G$.
\end{de}

Let us now give the definition of a C*-tensor category with duals which will make the link with categories of partitions clear.

\begin{de}\label{de:tensorcategory}
Let $V$ be a finite-dimensional Hilbert space. A \emph{(concrete) C*-tensor category with duals} $\mathfrak{C}$ is a collection of spaces $\Mor(w, w')\subset\B(V^{\otimes w}, V^{\otimes w'})$ for all words $w$ and $w'$ on $\{\circ, \bullet\}$ such that
\begin{enumerate}
\item If $T\in \Mor(w, w')$ and $T'\in \Mor(w'', w''')$, then $T\otimes T'\in \Mor(w.w'', w'.w''')$.
\item If $T\in \Mor(w, w')$ and $T'\in \Mor(w', w'')$, then $T'\circ T\in \Mor(w, w'')$.
\item If $T\in \Mor(w, w')$, then $T^{*}\in \Mor(w', w)$.
\item $\Id : x\mapsto x\in \Mor(\circ, \circ)$.
\item $D : x\otimes y \mapsto \langle x, y\rangle \in \Mor(\bullet\circ, \emptyset)$, where by convention $V^{\emptyset} = \CC$.
\end{enumerate}
\end{de}

Our statement of Tannaka-Krein duality is given in the fashion of \cite[Thm 3.6]{banica2009liberation} :

\begin{thm}[Woronowicz]\label{thm:tannaka}
Let $\mathfrak{C}$ be a C*-tensor category with duals. Then, there exists a compact matrix quantum group $(\G, u)$ such that the underlying space of $u$ is isomorphic to $V$ and, for any words $w$ and $w'$ on $\{\circ, \bullet\}$,
\begin{equation*}
\Mor_{\G}(u^{\otimes w}, u^{\otimes w'}) = \Mor(w, w').
\end{equation*}
Moreover, $(\G, u)$ is unique up to isomorphism.
\end{thm}

Note that if $\sqcup$ denotes the rotated version of the partition $\vert$ then, with the notation above, $T_{\sqcup} = D$. Thus, we see that any category of partitions gives rise to a concrete C*-tensor category with duals as soon as we choose an integer $N$. This is the idea behind the definition of partition quantum groups. However, coloured partitions may have more than two colours so that we first have to adapt the notation $u^{\otimes w}$ to several colours. Let $\A$ be a colour set and assume that we have fixed a representation $u^{x}$ for every $x\in \A$. Then, if $w = w_{1}\dots w_{k}$ is a word on $\A$, we set
\begin{equation*}
u^{\otimes w} = u^{w_{1}}\otimes \dots \otimes u^{w_{n}}.
\end{equation*}
We are now ready for the definition of partition quantum groups, which relies on a simple application of Theorem \ref{thm:tannaka}.

\begin{thm}\label{thm:partitionqg}
Let $\A$ be a colour set, let $\CC$ be an $\A$-coloured category of partitions and let $N$ be an integer. Then, there exists a unique compact quantum group $\G$ together with representations $(u^{x})_{x\in \A}$ of dimension $N$ whose direct sum is generating and such that
\begin{equation*}
\Mor_{\G}(u^{\otimes w}, u^{\otimes w'}) = \Span\left\{T_{p}, p\in \CC(w, w')\right\}.
\end{equation*}
\end{thm}

\begin{proof}
Let us first assume that $\A$ is finite. The idea is to get a compact quantum group whose fundamental representation is the sum of all the representations $u^{x}$ for $x\in \A$. To do this, we only have to define the morphism spaces of this $u$ and apply Theorem \ref{thm:tannaka}. Let us first notice that there are obvious isomorphisms
\begin{equation*}
\left\{\begin{array}{ccc}
\Mor(a\oplus b, c) & = & \Mor(a, c)\oplus\Mor(b, c) \\
\Mor(a, b\oplus c) & = & \Mor(a, b) \oplus \Mor(a, c)
\end{array}\right.
\end{equation*}
Thus, if $s$ is a word on $\{\circ, \bullet\}$, $u^{\otimes s}$ should be the sum of $u^{\otimes w}$ for all words $w$ on $\A$ having the same length as $s$. Accordingly we set, for any words $s$ and $s'$ on $\{\circ, \bullet\}$,
\begin{equation*}
\Mor(s, s') = \bigoplus_{\vert w\vert = \vert s\vert, \vert w'\vert = \vert s'\vert} \Span\left\{T_{p}, p\in \CC(w, w')\right\}
\end{equation*}
where $w$ and $w'$ are words on $\A$. Because $\CC$ is a category of partitions, this defines a \emph{bona fide} C*-tensor category with duals. By virtue of Theorem \ref{thm:tannaka}, this gives rise to a quantum group $\G$ together with a fundamental representation $u$. Let $T_{x}\in \Mor_{\G}(u, u)$ be the map associated to the $x$-identity. Since this is a multiple of a projection, it comes from a subrepresentation $u^{x}$ of $u$. Now, if $w = w_{1} \dots w_{n}$ and $w' = w'_{1}\dots w'_{k}$ are words on $\A$, then the space of morphisms between $u^{\otimes w}$ and $u^{\otimes w'}$ is
\begin{equation*}
T_{w}.\Mor(u^{\otimes k}, u^{\otimes l}).T_{w'},
\end{equation*}
where $T_{w} = T_{w_{1}}\otimes \dots\otimes T_{w_{n}}$ and similarly for $T_{w'}$. The assertion on the morphism spaces of tensor powers of the representations $u^{x}$ is then straightforward from this description. Let us now consider an infinite colour set $\A$. It can be written as an increasing union of finite colour sets $\A_{i}$. Each $\A_{i}$ gives rise to a compact quantum group $\G_{i}$ and the inclusions of colour sets turn the collection $(\G_{i})_{i}$ into an inductive system. It is then clear that the inductive limit quantum group satisfies the conclusion of the theorem.
\end{proof}

Note that the fact that the $x$-identity is in $\CC$ and the rotation operation imply that the contragredient representation of $u^{x}$ is $u^{\overline{x}}$, as expected.

\begin{de}
The compact quantum group $\G$ is called the \emph{partition quantum group} associated to the category of $\A$-coloured partitions $\CC$. If moreover $\CC$ is noncrossing, then $\G$ is said to be a \emph{noncrossing partition quantum group}.
\end{de}

\begin{rem}
One could slightly generalize this definition by dropping the requirement that the basic representations $u^{x}$ all have the same dimension. This requires to take a family of integers $(N_{x})_{x\in \A}$ satisfying some compatibility conditions depending on the category of partitions.
\end{rem}

As we will see in Section \ref{sec:examples}, there are many examples of partition quantum groups, including all known examples of free quantum groups and some classical groups whose Schur-Weyl duality can be implemented diagrammatically. The case of classical groups can be in fact easily identified, using the following definition : if $x, y\in \A$, the $(x, y)$-crossing is the partition $\crosspart = \{\{1, 4\}, \{2, 3\}\}$ with colours $x, y$ on the upper row and $y, x$ on the lower row.

\begin{prop}
Let $\CC$ be a category of partitions and let $N$ be an integer. Then, the associated partition quantum group is a classical group if and only if the $(x, y)$-crossing belongs to $\CC$ for all $x, y$ in $\A$ (including the $(x, x)$-crossings).
\end{prop}

\begin{proof}
Let $p$ be the $(x, y)$-crossing. The fact that $T_{p}$ is an intertwiner means that all the matrix coefficients of $u^{x}$ and $u^{y}$ commute to one another. Let $B$ be the subalgebra of $C(\G)$ generated by the matrix coefficients of all the representations $u^{x}$. Since the sum of these representations is generating, $B$ is dense in $C(\G)$. Thus, the assumptions implies that the C*-algebra $C(\G)$ is commutative, which by \cite[Thm 1.5]{woronowicz1987compact} implies that $\G$ is a classical compact group.
\end{proof}

We will postpone the study of examples to the dedicated Section \ref{sec:examples} in order to give a comprehensive treatment. For the moment, let us show how one can derive the representation theory of a partition quantum group from its category of partitions.

\subsection{Representation theory}\label{sec:representation}

The next step of our program is to study the representation theory of partition quantum groups in terms of their defining category of partitions. Part of this has been done in a joint work with M. Weber \cite{freslon2013representation} for the case of one colour. However, the proofs are valid in the general setting for the following reason : in all statements, we start by fixing a colouring $w$ and work in $V^{\otimes w}$. Thus, the colouring does not enter the picture any more. This has been detailed in the case of two colours in \cite[Sec 6]{freslon2013representation} and in \cite{freslon2013fusion}. Since the generalization to more colours is straightforward, we simply give a survey.

\subsubsection{General description}

In view of the link with Schur-Weyl duality, studying the representations of a partition quantum group $\G$ amounts to studying the representations of the "centralizer" algebras $\Mor(u^{\otimes w}, u^{\otimes w})$. This in turn can be formulated as a classification problem for idempotents of the latter algebra. It is therefore not surprising that the key tool is the so-called \emph{projective partitions}. Let us give all the necessary definitions for this at once.

\begin{de}
A partition $p$ is said to be \emph{projective} if $pp = p = p^{*}$. Moreover,
\begin{itemize}
\item A projective partition $p$ is said to be \emph{dominated} by another projective partition $q$ if $qp = p$. Then, $pq = p$ and we write $p\preceq q$.
\item Two projective partitions $p$ and $q$ are said to be \emph{equivalent} in a category of partitions $\CC$ if there exists $r\in \CC$ such that $r^{*}r = p$ and $rr^{*} = q$. We then write $p\sim q$ or $p\sim_{\CC} q$ if we want to keep track of the category of partitions.
\end{itemize}
\end{de}

Note that if $p$ and $q$ are equivalent in $\CC$, then they in fact both belong to $\CC$. It is clear that the linear operator $T_{p}$ is a scalar multiple of a projection if and only if $p$ is projective, giving us a way of building idempotents in $\Mor(u^{\otimes w}, u^{\otimes w})$. Moreover, $p\preceq q$ if and only if $T_{p}$ is dominated by $T_{q}$ as (multiples of) projections. More generally, recalling that a \emph{partial isometry} is an operator $T$ such that both $T^{*}T$ and $TT^{*}$ are projections, we have by \cite[Prop 2.18]{freslon2013representation},

\begin{prop}
For any partition $r$, $r^{*}r$ is a projective partition. In particular, all the operators $T_{r}$ are scalar multiples of partial isometries.
\end{prop}

\begin{de}
For a partition $r$, we set $T'_{r} = \|T_{r}\|^{-1/2}T_{r}$, which is a partial isometry. Note that $T'_{p}$ is a projection as soon as $p$ is a projective partition.
\end{de}

\begin{rem}
Our notations here differ from the ones introduced in \cite{freslon2013representation}, where $T'_{p}$ is denoted $T_{p}$ and $T_{p}$ is denoted $\mathring{T}_{p}$.
\end{rem}

To go further, let us introduce additional notations. For a word $w$ on $\A$, let $\Proj(w)$ denote the set of projective partitions in $\CC$ such that the upper (hence also lower) colouring of $p$ is $w$. For $p\in \Proj(w)$, we can define a projection $P_{p}\in \Mor(u^{\otimes w}, u^{\otimes w})$ and a representation $u_{p}\subset u^{\otimes w}\in C(\G)\otimes \B((\C^{N})^{\otimes \vert w\vert})$ by
\begin{equation*}
\left\{\begin{array}{ccc}
P_{p} & = & T'_{p} - \displaystyle\bigvee_{q\prec p} T'_{q} \\
u_{p} & = & (\ii\otimes P_{p})(u^{\otimes k})
\end{array}\right.
\end{equation*}
Here, the symbol $\bigvee$ denotes the supremum of projections. According to \cite[Thm 4.18 and Prop 4.22]{freslon2013representation}, the representations $u_{p}$ enjoy the following properties :
\begin{itemize}
\item If $p\sim q$ then $u_{p}\sim u_{q}$. If moreover we know that $u_{p}$ and $u_{q}$ are non-zero, then the converse holds.
\item $u^{\otimes w} = \displaystyle\sum_{p\in \Proj(w)} u_{p}$ in the sense that the only subrepresentation of $u^{\otimes w}$ containing all the $u_{p}$'s is $u^{\otimes w}$ itself.
\end{itemize}

This description has two drawbacks : it is hard to now whether $u_{p} \neq 0$ or not (this is closely linked to the \emph{Second fundamental theorem of invariants}) and the representation $u_{p}$ is not irreducible in general. The second point can be handled using representations of finite groups as explained in \cite[Prop 4.15]{freslon2013representation}. Since we will not need this construction hereafter, we it.
One can also give a combinatorial formula for tensor products of representations. For this, let us first give a structure result on projective partitions, which is a particular instance of \cite[Prop 2.9]{freslon2013representation}. The idea is to decompose a projective partition by cutting it in the middle.

\begin{prop}
Let $\A$ be a colour set and let us fix a colour $x_{0}\in \A$. Let $p\in \Proj(w)$ be a projective partition. Then, there is a unique partition $p_{u}\in P(w, x_{0}^{t(p)})$ such that
\begin{itemize}
\item $p = p_{u}^{*}p_{u}$.
\item The strings in $p_{u}$ between the upper and lower row do not cross each other.
\item All points of the lower row of $p_{u}$ are coloured with $x_{0}$.
\end{itemize}
\end{prop}

Now, given two projective partitions $p$ and $q$, one may mix them by inserting in the middle of $p\otimes q$ a partition collapsing some middle strings. Using the previous proposition, this means that we want to consider the partition
\begin{equation*}
p\ast_{h} q = (p_{u}^{*}\otimes q_{u}^{*})h(p_{u}\otimes q_{u})
\end{equation*}
for some suitable $h$, called a \emph{mixing partitions} (see \cite[Def 2.25]{freslon2013representation} for a pictorial description).

\begin{de}
Let $k$ and $l$ be integers. A \emph{$(k, l)$-mixing partition} is a projective partition $h\in P(x_{0}^{k+l}, x_{0}^{k+l})$ such that
\begin{itemize}
\item All blocks of $h$ have size $2$ or $4$.
\item Blocks of size $2$ are either of the form $(a, a')$ or $(a, b)$ with $a\leqslant k$ and $b > k$, or likewise $(a',b')$ with $a'\leqslant k'$ and $b' > k'$.
\item Blocks of size $4$ are of the form $(a, a', b, b')$ with $a\leqslant k$ and $b > k$.
\item All points of $h$ are coloured with $x_{0}$.
\end{itemize}
\end{de}

Then, the fusion rules are given by \cite[Thm 4.27]{freslon2013representation}, which reads :

\begin{thm}
Let $p$ and $q$ be projective partitions in $\CC$. Then,
\begin{equation*}
u_{p}\otimes u_{q} = u_{p\otimes q} + \sum_{h} u_{p\ast_{h}q},
\end{equation*}
where the sum runs over all $(t(p), t(q))$-mixing partitions $h$ and by convention, $u_{p\ast_{h}q} = 0$ if $p\ast_{h}q \neq \CC$.
\end{thm}

\subsubsection{The noncrossing case}\label{sec:noncrossing}

If the category of partitions $\CC$ is noncrossing, then the previous picture is greatly simplified and one gets a more tractable description of the representation theory. This gives a conceptual explanation for the relative "easiness" of the representation theory of noncrossing quantum groups compared to their classical analogues. More precisely, the representations $u_{p}$ enjoy the following additional properties :
\begin{itemize}
\item For every $p$, $u_{p}$ is non-zero and irreducible.
\item $u_{p}\sim u_{q}$ if and only if $p\sim q$.
\item $u^{\otimes w} = \displaystyle\bigoplus_{p\in \Proj(w)} u_{p}$ as a direct sum of representations.
\end{itemize}

The first two points are straightforward from the preceding section while the last one is a combination of \cite[Prop 4.22]{freslon2013representation} and \cite[Prop 3.7]{freslon2013fusion}. As for the fusion rules, they can be simplified in this context using the fact that there are few noncrossing mixing partitions. First, let $h_{\square}^{k}$ be the projective partition in $NC(x_{0}^{2k}, x_{0}^{2k})$ where the $i$-th point in each row is connected to the $(2k-i+1)$-th point in the same row (i.e. an increasing inclusion of $k$ blocks of size $2$) and all the points are coloured with $x_{0}$. If moreover we connect the points $1$, $k$, $1'$ and $k'$, we obtain another projective partition in $NC(x_{0}^{2k}, x_{0}^{2k})$ denoted by $h_{\boxvert}^{k}$. From this, we define binary operations on projective partitions (using $\vert$ to denote a suitably coloured version of the identity partition) :
\begin{equation*}
\left\{\begin{array}{ccc}
p\square^{k} q & = & (p_{u}^{*}\otimes q_{u}^{*})(\vert^{t(p)-k}\otimes h_{\square}^{k}\otimes \vert^{t(q)-k})(p_{u}\otimes q_{u}) \\
p\boxvert^{k} q & = & (p_{u}^{*}\otimes q_{u}^{*})(\vert^{t(p)-k}\otimes h_{\boxvert}^{k}\otimes \vert^{t(q)-k})(p_{u}\otimes q_{u})
\end{array}\right.
\end{equation*}
for $0 < k\leqslant \min(t(p), t(q))$. Using this, the fusion rules are given by \cite[Thm 6.8]{freslon2013representation} :
\begin{equation*}
u_{p}\otimes u_{q} = u_{p\otimes q}\oplus\bigoplus_{k=1}^{\min(t(p), t(q))} (u_{p\square^{k}q} \oplus u_{p\boxvert^{k}q}),
\end{equation*}
where by convention $u_{r} = 0$ if $r\notin \CC$.

\section{Examples}\label{sec:examples}

We will now give several examples of partition (quantum) groups. 

\subsection{Noncrossing quantum groups}

Partition quantum groups are mainly a generalization of \emph{easy quantum groups}, which were introduced by T. Banica and R. Speicher in \cite{banica2009liberation}, and can be recovered from our general definition :
\begin{itemize}
\item If $\A$ contains only one element, then $\G$ is an \emph{orthogonal easy quantum group} in the sense of \cite[Def 6.1]{banica2009liberation} and all such quantum groups arise in that way.
\item If $\A$ contains two colours which are inverse to one another, then $\G$ is a \emph{unitary easy quantum group} in the sense of \cite[Def 6.3]{freslon2013representation} and all such quantum groups arise in that way.
\end{itemize}

The free orthogonal, free unitary and free symmetric quantum groups introduced by A. van Daele and S. Wang in \cite{wang1995free}, \cite{wang1998quantum} and \cite{van1996universal} are particularly important examples of easy quantum groups. Orthogonal easy quantum groups are completely classified (see \cite{raum2013full}) and some partial results are known in the unitary case (in the forthcoming paper \cite{tarrago2013unitary}). We give here the list of free ones classified in \cite{freslon2013fusion} (see Section \ref{sec:fusion} for the definition of freeness) :
\begin{itemize}
\item The orthogonal ones $O_{N}^{+}$, $S_{N}^{+}$, $H_{N}^{+}$, $B_{N}^{+}$.
\item The free unitary quantum group $U_{N}^{+}$.
\item The \emph{free complexification} (see \cite{banica2007note} for a definition) $\widetilde{H}_{N}^{+}$ of $H_{N}^{+}$.
\item The quantum reflection groups $H_{N}^{s+} = \Z_{s}\wr_{\ast}S_{N}^{+}$. As we will see in the next subsection, this family has connections to classical complex reflection groups.
\end{itemize}

The last example is of peculiar importance since it is an example of a \emph{free wreath product} as defined by J. Bichon in \cite{bichon2004free}. More general free wreath products have been studied by F. Lemeux in \cite{lemeux2013fusion} and enter our setting. Let $\Gamma$ be a discrete group and let $\Gamma_{0}\subset \Gamma$ be a symmetric generating subset containing the neutral element. Let now $\A_{\Gamma_{0}}$ be $\Gamma_{0}$ seen as a colour set (the involution being given by the inverse in the group) and let $\CC_{\Gamma_{0}}$ be the category of $\A_{\Gamma_{0}}$-coloured partitions consisting in all noncrossing coloured partitions such that in each block, the product of the elements in the upper row (from left to right) is equal to the product of the elements in the lower row (from left to right). Then, \cite[Thm 2.20]{lemeux2013fusion} can be restated as

\begin{thm}[Lemeux]
Let $\Gamma$ be a discrete group, let $\Gamma_{0}\subset \Gamma$ be a symmetric generating subset containing the neutral element and let $N\geqslant 4$ be an integer. Then, the partition quantum group associated to $\CC_{\Gamma_{0}}$ and $N$ is the free wreath product $\widehat{\Gamma}\wr_{\ast}S_{N}^{+}$.
\end{thm}

Another nice feature of arbitrary colour sets is that they allow to make free products of partition quantum groups in the sense of \cite{wang1995free}. Assume that we have two colour sets $\A$ and $\A'$ and two categories of partitions $\CC$ and $\CC'$ coloured respectively with $\A$ and $\A'$. Consider the colour set $\A'' = \A\sqcup \A'$ and the category of $\A''$-coloured partitions $\CC'' = \CC\ast \CC'$ generated in $NC^{\A''}$ by $\CC$ and $\CC'$. Then, the following result is straightforward from \cite[Prop 2.15]{lemeux2013fusion} :

\begin{prop}\label{prop:freeproduct}
For any integer $N\geqslant 2$, let $\G$ and $\G'$ be the partition quantum groups associated to $\CC$ and $\CC'$ respectively. Then, the partition quantum group associated to $\CC''$ is the free product $\G\ast \G'$. 
\end{prop}

\subsection{Classical groups}\label{subsec:classical}

The first examples of partition groups are easy groups, of which the most important are :
\begin{itemize}
\item The orthogonal groups $O_{N}$ (this was first proved by R. Brauer in \cite{brauer1937algebras}), the bistochastic groups $B_{N}$, the hyperoctaedral groups $H_{N}$ (this is a special case of the work of K. Tanabe \cite{tanabe1997centralizer}) and the symmetric groups $S_{N}$ (this is a straightforward consequence of independent works of P. Martin \cite{martin1994temperley} and V.F.R. Jones \cite{jones1993potts}).
\item The unitary groups $U_{N}$ (this was first proved by P. Glockner and W. von Waldenfels in \cite{glockner1989relations}).
\item The complex reflection groups $H_{N}^{s} = G(s, 1, N)$ (this is again \cite{tanabe1997centralizer}).
\end{itemize}

Orthogonal easy groups were classified in \cite{banica2009fusion} while the classification in the unitary case will appear in \cite{tarrago2013unitary}. A simple way of building partition groups is to use \emph{abelianization}. More precisely, let $\G$ be any partition quantum group and consider the C*-algebra $C(\G)_{\text{ab}}$ which is the maximal abelian quotient of $C(\G)$. By \cite[Thm 1.5]{woronowicz1987compact}, this is a classical compact group. To see that it is a partition group, one simply notices that its morphism spaces are obtained from the morphism spaces of $\G$ by adding all the operators associated to the $(x, y)$-crossings for all $x, y\in \A$ (including the $(x, x)$-crossings). Here is the key fact concerning that construction :

\begin{prop}\label{prop:abelianization}
Let $G$ be a partition group. Then, there is a noncrossing partition quantum group $\G$ such that its abelianization is $G$.
\end{prop}

\begin{proof}
Let $\CC$ be the category of partitions of $G$ and set $\CC' = \CC\cap NC^{\A}$. Then, $\CC'$ is a category of noncrossing partitions. Moreover, it is clear that adding all the $(x, y)$-crossings to $\CC'$ enables to reconstruct any partition of $\CC$. Thus, if $N$ is an integer such that $G$ is associated to $\CC$ and $N$, then the abelianization of the quantum group $\G$ associated to $\CC'$ and $N$ is $G$.
\end{proof}

\begin{rem}
This correspondence is not one-to-one. For example, $B_{N}\times \Z_{2}$ can be obtained from both $B_{N}^{+}\times \Z_{2}$ and $B_{N}^{+}\ast \Z_{2}$.
\end{rem}

An easy consequence of this is the stability of the class of partition groups under direct products.

\begin{prop}
The direct product of a family of partition groups is again a partition group.
\end{prop}

\begin{proof}
According to Proposition \ref{prop:freeproduct}, the free product of the groups is a partition quantum group. Moreover, its abelianization, which is the direct product of the original groups, is also a partition group by  Proposition \ref{prop:abelianization}, hence the result.
\end{proof}

An important instance of the previous proposition is \cite[Prop 11.4]{banica2011free}, which asserts that the abelianization of the quantum reflection group $H_{N}^{s+}$ is the complex reflection group $H_{N}^{s}$. Let us restate this in terms of partitions. Let $\A = \{\circ, \bullet\}$ be the colour set where $\overline{\circ} = \bullet$ and let $\CC_{s, \text{ab}}$ be the category of all $\A$-coloured partitions $p$ such that in each block of $p$, the difference between the number of white and black points on each row is the same modulo $s$. Note that $\CC_{s, \text{ab}}$ can be obtained from the category of noncrossing partitions $\CC_{s}$ defined in \cite[Sec 4.5]{freslon2013fusion} by adding all the $(x, y)$-crossings for $x, y\in \{\circ, \bullet\}$.

\begin{prop}\label{prop:abelianizationwreath}
Let $N\geqslant 4$ be an integer. Then, the partition group associated to $\CC_{s, \text{ab}}$ and $N$ is the complex reflection group $H_{N}^{s}$.
\end{prop}

Consider the representation $v = u^{\circ}$ of $H_{N}^{s}$, which is simply the natural representation as unitary monomial matrices (i.e. having exactly one non-zero coefficient in each row and column). If we want to compute the intertwiners between $v^{\otimes k}$ and itself, we only need partitions with white points. The condition then becomes that the number of upper points in each block is equal to the number of lower points modulo $s$. We have therefore recovered K. Tanabe's result \cite[Lem 2.1]{tanabe1997centralizer}. Note that the associated centralizer algebras have been studied by M. Kosuda under the name of \emph{modular party algebras} in \cite{kosuda2008characterization}.

\begin{rem}
From this, it is tempting to look for a similar result for the more general family of unitary reflection groups $G(s, m, N)$. However, these are not partition groups, nor there exists a free analogue of them. This is a direct consequence of the classification of unitary easy groups in \cite{tarrago2013unitary}. This may be surprising in view of the work of K. Tanabe \cite{tanabe1997centralizer}. However, this works has two differences with our approach. First, the conditions defining the set $\Pi_{2k}(s, m, N)$ are not stable under horizontal concatenation and secondly it only involves the natural representation of $G(s, p, N)$ on $\C^{N}$ and not its contragredient representation.
\end{rem}

From this example of complex reflection groups, it is natural to go to more general wreath products.

\begin{prop}
Let $\Gamma$ be a finitely generated discrete group and let $N\geqslant 4$ be an integer. Then, the abelianization of $\h{\Gamma}\wr_{\ast}S_{N}^{+}$ is the usual wreath product $\h{\Gamma}_{\text{ab}}\wr S_{N}$, where $\h{\Gamma}_{\text{ab}}$ is the abelian compact group dual to the abelianization of $\Gamma$.
\end{prop}

\begin{proof}
This is in fact a particular case of a much more general statement from \cite{mori2012representation}, which will be explained in Example \ref{ex:classicalwreath}.
\end{proof}

\begin{rem}\label{rem:abelianproblem}
Here appears an important fact : our construction can only recover wreath products by \emph{abelian} compact groups. Another method allows us to obtain also wreath products by arbitrary \emph{finite} groups (see below). To go beyond these two cases, one needs a more general framework, which will be outlined in Subsection \ref{subsec:wreath}.
\end{rem}

Centralizers of wreath products by arbitrary finite groups where described by M. Bloss in \cite{bloss2003g}. However, his algebras of coloured diagrams cannot be realized as generalized partition algebras in our sense. To explain the construction, we will define new operators from the linear maps $T_{p}$, using an averaging technique. This is another way of writing the constructions of \cite{parvathi2004g}, giving both direct products and wreath products. Let $G$ be a finite group and set $\A_{G} = G$, \emph{but with $\overline{g} = g$}. If $p\in P^{\A_{G}}$ is an $\A_{G}$-coloured partition and if $g\in G$, we define $g.p$ to be the partition obtained by multiplying all the colours of $p$ by $g$ on the left.

We can now define new operators which are invariant under the action of $G$ by averaging. More precisely, we set, for $p\in P^{\A}$,
\begin{equation*}
L_{p} = \sum_{g\in G}T_{g.p}.
\end{equation*}
The reader can easily check that this definition corresponds to the operators $L$ defined in \cite{parvathi2004g}. Moreover, even though the operators $L_{p}$ do not span a generalized partition algebra, they still form a C*-tensor category with duals.

\begin{lem}\label{lem:directproduct}
Set $\Mor_{G\times}(k, l) = \Span\{L_{p}, p\in P^{\A_{G}}(k, l)\}$. Then, this yields a C*-tensor category with duals.
\end{lem}

\begin{proof}
We have to check that the morphism spaces are stable under all the operations. First the tensor product :
\begin{equation*}
\begin{array}{ccccc}
L_{p}\otimes L_{q} & = & \displaystyle\sum_{g, h\in G}T_{(g.p)\otimes (h.q)} & = & \displaystyle\sum_{g, h\in G}T_{g.(p\otimes (g^{-1}h.q))} \\
& = & \displaystyle\sum_{g\in G}\left(\displaystyle\sum_{h\in G} T_{g.(p\otimes (g^{-1}h.q))}\right) & = & \displaystyle\sum_{g\in G}\left(\displaystyle\sum_{k\in G} T_{g.(p\otimes (k.q))}\right) \\
& = & \displaystyle\sum_{k\in G}\left(\displaystyle\sum_{g\in G} T_{g.(p\otimes (k.q))}\right) & = &\displaystyle\sum_{k\in G} L_{p\otimes (k.q)}
\end{array}
\end{equation*}
Then the composition : let us assume that for any $k\in G$, $p$ and $k.q$ are not composable. Then, $L_{p}\circ L_{q} = 0$ by definition. Otherwise, there is a unique $k$ such that $p$ and $k.q$ are composable, and
\begin{equation*}
L_{p}\circ L_{q} = \sum_{g, h\in G}T_{g.p}\circ T_{h.q} = \sum_{g\in G} T_{g.(p(k.q))} = L_{p(k.q)}.
\end{equation*}
For the adjoint, it is clear that $L_{p}^{*} = L_{p^{*}}$. Moreover, the fact that the inverse of the colour $g$ is $g$ makes it clear that the rotation of partitions induces a duality on the category.
\end{proof}

By Theorem \ref{thm:tannaka}, this construction gives rise to a compact group and \cite{parvathi2004g} precisely asserts that this group is the direct product $G\times S_{N}$. To go to wreath product, we need to perform a second averaging. If $b(p)$ denotes the number of blocks of a partition $p$, we can order these blocks from left to right according to the position of their leftmost upper point. Then, any $b(p)$-tuple of elements of $g$ acts on $p$, the $k$-th element of the tuple acting on the $k$-th block only. We accordingly set
\begin{equation*}
M_{p} = \sum_{(g_{1}, \dots, g_{b(p)})\in G^{b(p)}}T_{(g_{1}, \dots, g_{b(p)}).p}. 
\end{equation*}
Again, these maps behave nicely.

\begin{prop}
Set $\Mor_{G\wr}(k, l) = \Span\{M_{p}, p\in P^{\A_{G}}(k, l)\}$. Then, this yields a C*-tensor category with duals.
\end{prop}

\begin{proof}
The computations are similar to the ones of for Lemma \ref{lem:directproduct}, we therefore omit them.
\end{proof}

This time, the compact group associated to this category is, by \cite[Thm 4.1.4]{parvathi2004g} and \cite[Thm 6.6]{bloss2003g}, the wreath product $G\wr S_{N}$ we were looking for. Note that since
\begin{eqnarray*}
\frac{1}{\vert G\vert}\sum_{(g_{1}, \dots, g_{b(p)})\in G^{b(p)}}L_{(g_{1}, \dots, g_{b(p)}).p} & = & \frac{1}{\vert G\vert}\sum_{g\in G}\sum_{(g_{1}, \dots, g_{b(p)})\in G^{b(p)}}L_{(g_{1}g, \dots, g_{b(p)}g).p} \\
& = & \frac{1}{\vert G\vert}\sum_{g\in G}M_{p} = M_{p},
\end{eqnarray*}
we have an inclusion of algebras $\Mor_{G\wr}\subset \Mor_{G\times}\subset \Mor_{P^{\A_{G}}}$ which translates into a reversed inclusion of groups $S_{N}\subset G\times S_{N}\subset G\wr S_{N}$.

\subsection{Generalizations}

The constructions of the previous subsection make sense in a broader setting. Let $\CC$ be a category of \emph{uncoloured} partitions and let $G$ be a finite group. Then, one can build a C*-tensor category with duals $\mathfrak{C}(\CC, G, \times)$ by setting :
\begin{equation*}
\Mor_{\mathfrak{C}(\CC, G, \times)}(k, l) = \Span\{L_{p}, p\in \CC\}.
\end{equation*}
That this is a C*-tensor category with duals is proved in the same way as Lemma \ref{lem:directproduct}. The associated compact quantum group can easily be identified, thus extending the results of \cite{parvathi2004g}.

\begin{prop}
Let $N$ be an integer and let $\G$ be the partition quantum group associated to $\CC$. Then, the compact quantum group associated to $\mathfrak{C}(\CC, G, \times)$ is the direct product $G\times \G$.
\end{prop}

\begin{proof}
Consider the quantum group $\G' = G\times \G$, let $u$ be the fundamental representation of $\G$ given by the partition quantum group construction and consider the representation $v = \lambda\otimes u$ of $\G'$, where $\lambda$ is the regular representation of $G$. Let $V$ be the carrier space of $u$, let $(e_{i})_{1\leqslant i\leqslant N}$ be a basis of $V$. For each $g\in G$, let us denote by $V^{g}$ a copy of $V$ with basis $(e_{i}^{g})_{1\leqslant i\leqslant N}$ and by $u^{g}$ the natural copy of $u$ acting on this space. Then, $u' = \oplus u^{g}$ is the image of $\Id \otimes u$ under the isomorphism
\begin{equation*}
\ell^{2}(G)\otimes V \longrightarrow \bigoplus_{g\in G}V^{g}
\end{equation*}
given by $e_{i}\otimes\delta_{g}\mapsto e_{i}^{g}$. Moreover, the image of $\lambda\otimes \Id$ is the representation $\rho$ permuting the spaces $V^{g}$. This yields a bijection between the intertwiners of tensor powers of $\lambda\otimes u$ and the corresponding intertwiners of $u'\rho = \rho u'$. Because $u'$ and $\rho$ commute, the latter intertwiner spaces are the intersection of the intertwiner spaces of $u'$ and $\rho$. By Lemma \ref{lem:directproduct}, these spaces are linearly spanned by the maps $L_{p}$. Since $v$ is a generating representation for $G\times \G$, the latter is the quantum group associated to $\mathfrak{C}(\CC, G, \times)$
\end{proof}

Similarly, we can define another C*-tensor category with duals $\mathfrak{C}(\CC, G, \wr)$ using the maps $M_{p}$. However, the resulting quantum group does not seem to be easily described using $G$ and the quantum group associated to $\CC$. Of course, if $\CC$ is the set of all partitions, then we get the wreath product $G\wr S_{N}$ by the result of M. Bloss \cite{bloss2003g}. But if $\CC$ is for instance the category of all pair partitions (corresponding to $O_{N}$), then there is no natural candidate for a "permutational wreath product" by the orthogonal group. On the quantum side the first natural example to consider is $\mathfrak{C}(NC, \Z_{2}, \wr)$. By \cite[Thm 4.4]{banica2011two}, this is the category of representations of the quantum group $S^{+}(N, 0)$, which is proven in \cite[Thm 2.9]{banica2011two} to be isomorphic to $H_{N}^{+} = \Z_{2}\wr_{\ast}S_{N}^{+}$. It is of course very tempting to conjecture from this that $\mathfrak{C}(NC, G, \wr)$ is the representation category of $G\wr_{\ast} S_{N}^{+}$ for any finite group $G$, although we have no further evidence for this. Quantum isometry groups of free products give more examples of this situation. To explain this, let us note that the construction of $\mathfrak{C}(\CC, G, \wr)$ still makes sense if the partitions in $\CC$ are already coloured by $\A_{G}$. In the special case $G = \Z_{2}$, we can identify this colouring with the black/white colouring used in \cite{banica2011two} and \cite{banica2012quantum}. Using the description of the two-coloured categories of partitions of the quantum reflection groups given in \cite{freslon2013fusion}, \cite[Thm 3.5]{banica2012quantum} translates in the following way :

\begin{thm}[Banica-Skalski]
Assume that $5 \leqslant s < \infty$ and let $\CC_{s}$ denote the category of partitions of $H_{N}^{s+}$. Then $\mathfrak{C}(\CC_{s}, \Z_{2}, \wr)$ is the category of representations of the quantum isometry group of the dual of $\Z_{s}^{\ast N}$.
\end{thm}

\begin{rem}
In \cite{banica2012quantum}, the inverse in the colour set is given by exchanging the colours, contrary to our previous assumption. However, the averaging process $p\mapsto M_{p}$ still yields a C*-tensor category with duals because the group acting is $\Z_{2}$. As soon as $G$ contains an element of order more than $2$, the existence of a duality becomes unclear.
\end{rem}

The link between these quantum isometry groups and the quantum reflection groups is still unclear as far as we know. Understanding it would certainly be an important step in the study of these "generalized wreath products". We will come back to this problem in a more general setting in Section \ref{subsec:wreath}.

\section{Free fusion semirings}\label{sec:fusion}

We will now concentrate on quantum groups whose fusion semiring satisfies some freeness assumption. Let us first fix some terminology and notations. Let $\G$ be a compact quantum group and let $\Irr(\G)$ be the set of equivalence classes of irreducible representations of $\G$. The \emph{fusion semiring} $R^{+}(\G)$ of $\G$ is the set $\N[\Irr(\G)]$ endowed with the operations induced by the direct sum and tensor product of representations. As explained in Subsection \ref{sec:noncrossing}, the representation theory of a compact quantum group can be completely described when it comes from a category of noncrossing partitions. A systematic approach to the computation of fusion rings for quantum groups associated to noncrossing partitions with two colours was given in \cite{freslon2013fusion}, together with some classification results. We will now see how this generalizes to an arbitrary colour set. We first have to clarify the notion of freeness.

\subsection{Fusion sets and freeness}

Free fusion semirings were introduced in \cite{banica2009fusion} as an attempt to unify several examples of "free" quantum groups, as well as to isolate the crucial features in the proof of some operator algebraic properties of these quantum groups (as illustrated in the Appendix). In order to make things more clear in the sequel, we first give a proper definition of the basic data needed for the construction of a free fusion ring.

\begin{de}
A \emph{fusion set} is a set $S$ together with a conjugation map $x\mapsto \overline{x}$ which is involutive and a fusion operation
\begin{equation*}
\ast : S\times S \rightarrow S\cup\{\emptyset\}.
\end{equation*}
\end{de}

\begin{ex}\label{ex:groupoid}
A group is a simple example of a fusion set, the conjugation being the inverse and the fusion operation being the group law. A more general example is given by groupoids : let $\Gr$ be a groupoid and let $S_{\Gr}$ be the set of all morphisms of $\Gr$. For $x, y\in S_{\Gr}$, set $\co{x} = x^{-1}$ and $x\ast y = x\circ y$ if the composition makes sense and $\emptyset$ otherwise. Then, $S_{\Gr}$ is a fusion set. As we will see in Corollary \ref{cor:groupoiddecomposition}, this example is almost generic when considering fusion sets arising from compact quantum groups.
\end{ex}

Let $S$ be a fusion set and consider the free monoid $F(S)$ on $S$, i.e. the set of all words on $S$. The operations on $S$ extend to the abelian semigroup $\N[F(S)]$ in the following way : if $w_{1}\dots w_{n}, w'_{1}\dots w'_{k} \in F(S)$, then
\begin{itemize}
\item $\overline{w_{1}\dots w_{n}} = \overline{w}_{n}\dots \overline{w}_{1}$
\item $(w_{1}\dots w_{n})\ast (w'_{1}\dots w'_{k}) = w_{1}\dots (w_{n}\ast w'_{1})\dots w'_{k}$
\end{itemize}
the latter being set equal to $0$ whenever one of the two words is empty, or if $w_{n}\ast w_{1}' = \emptyset$. We can now turn $\N[F(S)]$ into a \emph{semiring} $(R^{+}(S), +, \otimes)$ by setting
\begin{equation*}
w\otimes w' = \sum_{w = az, w' = \overline{z}b} ab + a\ast b.
\end{equation*}

\begin{de}
A semiring $R^{+}$ is said to be \emph{free} if there exists a fusion set $S$ such that $R^{+}$ is isomorphic to $R^{+}(S)$. A compact quantum group $\G$ is said to be free if $R^{+}(\G)$ is free.
\end{de}

\begin{rem}
The term "free" has been used with various meanings in the literature, sometimes meaning "associated to a noncrossing category of partitions". We will rather call the latter "noncrossing quantum groups" and keep the word "free " for quantum groups having free fusion semiring, regardless of the partition setting.
\end{rem}

\begin{rem}
Later on, we will classify all fusion sets coming from compact quantum groups, so that it will appear that there are infinitely many non-isomorphic free quantum groups. However, our definition may not be very practical when dealing with the converse problem : given a semiring $R^{+}$, is there a criterion ensuring its freeness ? We will leave this question aside in this work, as well as the following closely related one : if $S_{1}$ and $S_{2}$ are fusion sets such that $R^{+}(S_{1}) \simeq R^{+}(S_{2})$, do we have $S_{1}\simeq S_{2}$ ? 
\end{rem}

The link with noncrossing partition quantum groups relies on the definition of the fusion set associated to a category of partitions $\CC$. Let $S(\CC)$ be the set of equivalence classes of projective partitions with only one block. This means that we are looking at partitions $p$ such that all the points are connected. If $p$ is such a partition, then we define $\co{p}$ to be the partition obtained by rotating all the points of $p$ upside down. We can define two operations on $S(\CC)$ :
\begin{itemize}
\item $\co{[p]} = [\co{p}]$
\item $[p]\ast[q] = \left\{\begin{array}{cc}
[p\boxvert q] & \text{if } p\boxvert q\in \CC \\
\emptyset & \text{otherwise}
\end{array}\right.$
\end{itemize}
That these operations are well-defined was proved in \cite[Lem 4.13 and Lem 4.14]{freslon2013fusion}. We thus have defined a fusion set, hence a fusion semiring.

\begin{de}
Let $\CC$ be a category of partitions. Its fusion semiring $R^{+}(\CC)$ is the fusion semiring of the fusion set $S(\CC)$.
\end{de}

Recall from Subsection \ref{sec:noncrossing} that if $\CC$ is noncrossing, then the irreducible representations of the associated quantum group are indexed by projective partitions. Intuitively, a projective noncrossing partition can be thought of as a word on one-block projective partitions (here noncrossingness is crucial), so that $R^{+}(\G)$ should be closely linked to $R^{+}(\CC)$. Formally, there is a natural map
\begin{equation*}
\Phi : R^{+}(\CC) \longrightarrow R^{+}(\G)
\end{equation*}
sending a word $[p_{1}]\dots [p_{n}]$ to $[u_{p_{1}\otimes \dots \otimes p_{n}}]$ and extended by linearity and we want a criterion ensuring its bijectivity and compatibility with the tensor product. It is not difficult to see (for instance in \cite[Rem 4.4]{raum2012isomorphisms}) that $R^{+}(\G)$ cannot be free as soon as $\G$ has a non trivial one-dimensional representation. The converse was proved in \cite[Thm 4.18]{freslon2013fusion} for two colours, but the proof carries on verbatim to an arbitrary colour set.

\begin{thm}\label{thm:freefusion}
Let $\CC$ be a category of noncrossing coloured partitions, let $N\geqslant 4$ be an integer and let $\G$ be the associated partition quantum group. Then, the following are equivalent :
\begin{enumerate}
\item The map $\Phi$ is an isomorphism.
\item The fusion semiring $R^{+}(\G)$ is free.
\item $\G$ has no non trivial one-dimensional representation.
\item The category of partitions $\CC$ is \emph{block-stable}, i.e. for any $p\in \CC$ and any block $b$ of $p$, we have $b\in \CC$.
\end{enumerate}
\end{thm}

\begin{rem}
In \cite{freslon2013fusion} we classified the possible groups of one-dimensional representations of any noncrossing quantum group on two mutually inverse colours. It would of course be interesting to make a similar study in this more general context. However, the situation will be more complicated since several combinatorial tricks used in \cite{freslon2013fusion} (e.g. reducing to the case of partitions with only one colour) are not available any more. Moreover, it is likely that any discrete group may appear as the group of one-dimensional representations of a noncrossing partition quantum group, even though the "free part" gives restrictions. 
\end{rem}

\begin{rem}\label{rem:haagerup}
The argument of \cite[Thm 6.11]{freslon2013fusion} applies without change to this broader setting, so that if $\G$ is a free partition quantum group with $S$ finite, then $\G$ has the Haagerup property (see \cite{daws2014haagerup} for the definition).
\end{rem}

\subsection{Partitions and freeness}\label{subsec:partitions}

The aim of this section is to give a comprehensive description of the free fusion semirings arising from the partition quantum group construction. In fact, we will prove in Theorem \ref{thm:fullclassification} that \emph{any free fusion semiring arising from a quantum group is the fusion semiring of a noncrossing partition quantum group}. The strategy is to associate to any free fusion semiring a category of noncrossing partitions. There is an \emph{a priori} natural way of doing this : let $S$ be a fusion set and consider the colour set $\A_{S} = S$ where the inverse is given by the involution of $S$. Let us assume that $S$ is \emph{associative}, so that we can define a function $f : F(S) \rightarrow S$ by
\begin{equation*}
f(w) = w_{1}\ast\dots \ast w_{n}
\end{equation*}
for any word $w = w_{1}\dots w_{n}$. If $p\in P^{\A_{S}}(w, w')$ is the one-block partition with upper colouring $w$ and lower colouring $w'$, we set $f_{u}(p) = f(w)$ and $f_{l}(p) = f(w')$. The idea is to consider the following partitions : 

\begin{de}
Let $S$ be an \emph{associative} fusion set.
\begin{itemize}
\item If $p$ is a one-block partition in $NC^{\A_{S}}(w, w')$ with $w, w'\neq \emptyset$, we say that $p$ is \emph{$f$-invariant} if $f_{u}(p) = f_{l}(p)$.
\item If $p$ is a one-block partition in $NC^{\A_{S}}(w, \emptyset)$, we say that $p$ is \emph{$f$-invariant} if $f_{u}(p) = \co{x}\ast x$ for some $x\in S$.
\end{itemize}
\end{de}

We would like to consider the category of partitions whose blocks are all $f$-invariant. However, problems arise when considering elements $x\in S$ such that $\co{x}\ast x = \emptyset$ (for example the fundamental representation of $O_{N}^{+}$) since the rotation of the $x$-identity is not $f$-invariant in that case. We will therefore split $S$ into two parts to define the category of partitions. Let $\HH(S)$ be the set of elements of $S$ such that $\co{x}\ast x\neq \emptyset$.

\begin{de}
For any two words $w$ and $w'$ on $\HH(S)$, we define $\CC_{\HH(S)}(w, w')$ to be the set of all partitions in $NC^{\A_{\HH(S)}}(w, w')$ such that all blocks are $f$-invariant. Let $\CC_{S}'$ be the category of partitions generated by the $x$-identity for all $x\notin\HH(S)$. Then, we define $\CC_{S}$ to be the free product category of partition $\CC'_{S}\ast\CC_{\HH(S)}$.
\end{de}

In the remainder of this section, we will prove that $\CC_{S}$ is a category of partitions yielding the same fusion semiring as $S$. This is of course false for an arbitrary fusion set since $\CC_{S}$ need not even be a category of partitions in general. We therefore first have to find extra properties satisfied by fusion sets $S$ arising from compact quantum groups. Then, we will prove that these properties are enough to get an isomorphism of fusion set between $S(\CC_{S})$ and $S$.

\begin{de}\label{de:associative}
Let $S$ be a fusion set. It is said to be
\begin{itemize}
\item \emph{Associative} if $\ast$ is associative as a law on $S\cup \{\emptyset\}$ (with the convention that $x\ast\emptyset = \emptyset\ast x = \emptyset$ for any $x\in S\cup\{\emptyset\}$).
\item \emph{Frobenius} if for any $x, y, z\in S$, $x = y\ast z \Leftrightarrow \co{x}\ast y = \co{z}$.
\item \emph{Antisymmetric} if for any $x, y\in S$, $\co{x\ast y} = \co{y}\ast\co{x}$.
\end{itemize}
\end{de}

Of course, these properties need not be satisfied by arbitrary fusion sets. However, our interest lies in those coming from compact quantum groups, to which we give a name for convenience.

\begin{de}
A fusion set $S$ is said to be \emph{admissible} if there exists a compact quantum group $\G$ such that $R^{+}(\G) \simeq R^{+}(S)$.
\end{de}

The link with the properties above is provided by the following proposition :

\begin{prop}
Let $S$ be an admissible fusion set. Then, $S$ is associative, Frobenius and antisymmetric.
\end{prop}

\begin{proof}
Let $\G$ be a compact quantum group such that $R^{+}(\G)\simeq R^{+}(S)$. For any $x$ in $S$, choose an irreducible unitary representation $u^{x}$ of $\G$ such that $[u^{x}]$ is the image of $x$ under the above isomorphism. Consider three elements $x, y, z\in S$ and compute their triple tensor product in two ways ($\varepsilon$ denoting the trivial representation) :
\begin{eqnarray*}
u^{x}\otimes (u^{y}\otimes u^{z}) & = & u^{x}\otimes (u^{yz}+u^{y\ast z} + \delta_{y=\co{z}}\varepsilon) \\
 & = & u^{xyz}+u^{x\ast(yz)}+\delta_{x=\co{y}}u^{z}+u^{x(y\ast z)}+u^{x\ast (y\ast z)}+\delta_{x = \co{y\ast z}}\varepsilon +\delta_{y=\co{z}}u^{x}
\end{eqnarray*}
and
\begin{eqnarray*}
(u^{x}\otimes u^{y})\otimes u^{z} & = & (u^{xy}+u^{x\ast y}+\delta_{x=\co{y}}\varepsilon)\otimes u^{z} \\
 & = & u^{xyz}+u^{(xy)\ast z}+\delta_{y=\co{z}}u^{x}+u^{x(y\ast z)}+u^{(x\ast y)\ast z}+\delta_{x\ast y = \co{z}}\varepsilon +\delta_{x=\co{y}}u^{z}
\end{eqnarray*}
Recall that by convention, $u^{w\ast w'}$ is zero if $w\ast w' = \emptyset$. Note that by definition of the fusion operation on $R^{+}(S)$, $x\ast(yz) = (x\ast y)z$ and $x(y\ast z) = (xy)\ast z$. These simplifications together with the associativity of the tensor product yield
\begin{equation}\label{eq:associativity}
u^{x\ast (y\ast z)} + \delta_{x = \co{y\ast z}}\varepsilon = u^{(x\ast y)\ast z}+\delta_{x\ast y = \co{z}}\varepsilon.
\end{equation}
Assume now that $y\ast z \neq\emptyset$ and that $x\ast(y\ast z)\neq\emptyset$. Then, the left-hand side contains a non-trivial representation and therefore $x\ast y\neq\emptyset$ and $(x\ast y)\ast z = x\ast (y\ast z)$. Reciprocally, if one of these elements is equal to $\emptyset$, then both sides cannot contain a non trivial representation, so that the other side must be equal to $\emptyset$ too. This proves the associativity of $S$. Because of this associativity, both sides of Equation \eqref{eq:associativity} can contain at most one copy of the trivial representation. Thus, $x = \co{y\ast z}$ if and only if $x\ast y = \co{z}$. Replacing $x$ by $\co{x}$ gives the Frobenius property. Eventually, $S$ is antisymmetric because for any two representations $u$ and $v$ of a compact quantum group, the contragredient representation of $u\otimes v$ is unitarily equivalent to $\co{v}\otimes \co{u}$. Applying this to $u^{x}$ and $u^{y}$, together with the fact that $\co{xy} = \co{y}\:\co{x}$, gives the desired equality.
\end{proof}

The next step is to check that for an admissible $S$, $\CC_{\HH(S)}$ is indeed a category of partitions. The subtle part is the stability under vertical concatenation, for which we will use a separate lemma for the sake of clarity. We first need a useful computation.

\begin{lem}\label{lem:inverse}
Let $S$ be an admissible fusion set and consider two elements $y, z\in S$ such that $y\ast z\neq \emptyset$. Then, $\co{y}\ast y\ast z = z$.
\end{lem}

\begin{proof}
Recall that the Frobenius property reads $x = y\ast z \Leftrightarrow \co{x}\ast y = \co{z}$. Replacing $x$ by $y\ast z$ in the second equality and using antisymmetry yields $\co{z}\ast\co{y}\ast y = \co{z}$. Applying the involution and using antisymmetry again, we get the result.
\end{proof}

Understanding the structure of blocks in the vertical concatenation is a difficult problem in general. We will take advantage of the noncrossingness to decompose this operation using the following notion of pseudo-one-block partition.

\begin{de}
A partition $p$ is said to be \emph{pseudo-one-block} if all its blocks but one are identity partitions.
\end{de}

\begin{lem}\label{lem:pseudooneblock}
Let $q$ be a noncrossing coloured partition. Then, $q$ is a vertical concatenation of pseudo-one-block partitions.
\end{lem}

\begin{proof}
Let $b$ be a non-through-block in the upper row of $q$ which is an interval in the sense that any point between two points of $b$ is again in $b$. Then, tensoring $b$ by suitable identity partitions we get a partition $b'$ such that $q = q'b'$, where $q'$ is the partition obtained by removing $b$ in $q$. Iterating this process, we can remove all the upper non-through-blocks of $q$ and get a decomposition $q = pb'_{1}\dots b'_{k}$, where $p$ has no non-through-block in the upper row. The same decomposition can be done in the lower row, yielding another decomposition
\begin{equation*}
q = (b''_{1}\dots b''_{l})r(b'_{1}\dots b'_{k}).
\end{equation*}
Here, all the partitions are pseudo-one-block except for $r$, which is the partition obtained from $q$ by keeping only the through-blocks. Let us denote by $(q_{i})_{1\leqslant i\leqslant n}$ these through-blocks ordered from left to right. Because they do not cross, we have $q = q_{1}\otimes \dots\otimes q_{n}$ and it is clear that this partition is a vertical concatenation of pseudo-one-bock partitions, hence the result.
\end{proof}

We are now ready to prove the stability of $\CC_{\HH(S)}$ under vertical concatenation.

\begin{lem}\label{lem:verticalconcatenation}
Let $p, q\in \CC_{\HH(S)}$. Then, $qp\in \CC_{\HH(S)}$.
\end{lem}

\begin{proof}
Because of Lemma \ref{lem:pseudooneblock}, we can assume that $q$ is a pseudo-one-block partition. Let us consider a block $b$ in $qp$ and prove that $b\in \CC_{\HH(S)}$. There are three possibilities :
\begin{itemize}
\item $b$ is a through-block
\item $b$ is a non-through-block which was not in $p$
\item $b$ is a non-through-block of $p$
\end{itemize}
Note that in the third case, $b\in \CC_{\HH(S)}$ by definition. Let us assume that $b$ is a through-block and let $c_{1}, \dots, c_{k}$ be the through-blocks in $p$ containing upper points of $b$. Then, $f_{u}(b) = f_{u}(c_{1})\ast\dots \ast f_{u}(c_{k})$. Because the $c_{i}$'s are blocks of $p$, they belong to $\CC_{\HH(S)}$ and we have
\begin{equation*}
f_{u}(b) = f_{l}(c_{1})\ast\dots \ast f_{l}(c_{k}).
\end{equation*}

Let $d$ be the unique non-identity block in $q$. Then, $b$ is obtained in the following way :
\begin{enumerate}
\item Connect all the points in the lower row of $p$ between the extreme left point of $c_{1}$ and the extreme right point of $c_{k}$ to obtain a block $b'$.
\item If $d\in \CC(w, w')$ is a through-block, substitute $w$ to $w'$ in the lower row.
\item If $d\in \CC(w, \emptyset)$, cancel $w$ in the lower row.
\item If $d\in \CC(\emptyset, w')$, do not change anything.
\end{enumerate}
Step $(1)$ has the effect of inserting the colourings of some non-through-block partitions between the lower colourings of the blocks $c_{i}$. However, these are $f$-invariant so that by Lemma \ref{lem:inverse}, $f_{l}(b') = f_{u}(b)$. Since moreover $d\in \CC_{\HH(S)}$, $f(w) = f(w')$ and by associativity, step $(2)$ does not change the value of the map $f_{l}$ so that in the end, $f_{l}(b) = f_{l}(b') = f_{u}(b)$. The same holds for step $(3)$ since in that case, $f(w) = \co{x}\ast x$ and removing it does not change the value of $f$ by Lemma \ref{lem:inverse}.

If $b$ is a non-through-block which is not in $p$, $d$ must be a non-through partition which cancels the whole lower row of $b'$. The fact that such a partition matching the colouring exists means that $f_{l}(b') = \co{x}\ast x$ for some $x\in S$. Then, $f_{u}(b) = f_{l}(b') = \co{x}\ast x$ and $b$ is $f$-invariant.
\end{proof}

\begin{lem}\label{lem:category}
Let $S$ be an admissible fusion set. Then, $\CC_{\HH(S)}$ is a block-stable category of partitions. As a consequence, $\CC_{S}$ is also a block-stable category of partitions.
\end{lem}

\begin{proof}
First note that the associativity of $S$ is needed even simply to define $\CC_{\HH(S)}$. We then have to check the stability of $\CC_{\HH(S)}$ under the category operations.
\begin{itemize}
\item \emph{Horizontal concatenation} : this is clear since any block of $p\otimes q$ is either a block of $p$ or a block of $q$.
\item \emph{Vertical concatenation} : this was proved in Lemma \ref{lem:verticalconcatenation}.
\item \emph{Adjoint} : it is enough to check it block-wise, where it is a direct consequence of the antisymmetry of $S$.
\item \emph{Rotation} : again we can check it block-wise, where it is a direct consequence of the Frobenius property of $S$.
\end{itemize}
Thus, $\CC_{\HH(S)}$ is a category of partitions and is by definition block-stable. Since a free product of block-stable categories of partitions is again block-stable, the proof is complete.
\end{proof}

The isomorphism between $S(\CC_{S})$ and $S$ will be implemented by the restriction of the map $f_{u}$ to one-block projective partitions (where it coincides with $f_{l}$). This requires $f_{u}$ to be compatible with the equivalence relation on projective partitions.

\begin{lem}\label{lem:fullclassification}
Let $S$ be an admissible fusion set and let $p, q\in \CC_{S}$ be one-block projective partitions. Then, $p\sim q$ if and only if $f_{u}(p) = f_{u}(q)$.
\end{lem}

\begin{proof}
First remark that one-block projective partitions in $\CC'_{S}$ are simply identity partitions, so that they are always $f$-invariant. Consider now the partition $r = r^{p}_{q} = q_{u}^{*}p_{u}$ and recall that $p$ is equivalent to $q$ if and only if $r^{p}_{q}\in \CC_{S}$. Note that $r$ has only one block because $p$ and $q$ do. Thus, since $f_{u}(r^{p}_{q}) = f_{u}(p)$ and $f_{l}(r^{p}_{q}) = f_{u}(q)$, $r^{p}_{q}\in \CC_{S}$ if and only $f_{u}(p) = f_{u}(q)$.
\end{proof}

\begin{thm}\label{thm:fullclassification}
Let $\G$ be a free compact quantum group. Then, there exists a partition quantum group $\G'$ such that $R^{+}(\G)\simeq R^{+}(\G')$.
\end{thm}

\begin{proof}
Let $S$ be a fusion set such that $R^{+}(\G) = R^{+}(S)$ and let $\CC_{S}$ be the associated category of noncrossing partitions. By Lemma \ref{lem:fullclassification}, the map $f_{u}$ restricts to a set-theoretic isomorphism between $S(\CC_{S})$ and $S$. It is clear that $f$ respects the involution (because $S$ is antisymmetric). For the fusion operation, note that if $x\in S$ is such that $x\ast y\neq\emptyset$, then by Lemma \ref{lem:inverse}, $\co{x}\ast x\neq\emptyset$ and $x\in \HH(S)$. In other words, there is no possible fusion involving an element which is not in $\HH(S)$. Since $f$ respects the fusion operation on $S(\CC'_{S})$ and $S(\CC_{\HH(S)})$, it therefore respects the fusion operation on $S$. Thus, it is an isomorphism of fusion sets and $R^{+}(\G) \simeq R^{+}(\CC_{S})$. By Lemma \ref{lem:category}, $\CC_{S}$ is a block-stable so that we can conclude by Theorem \ref{thm:freefusion} that $R^{+}(\G) = R^{+}(\G')$, where $\G'$ is a partition quantum group associated to $\CC_{S}$ and any integer $N\geqslant 4$.
\end{proof}

In the language of \cite{banica1999fusion}, any free compact quantum group is an \emph{$R^{+}$-deformation} of a free partition quantum group. A natural question is therefore : is any free compact quantum group \emph{monoidally equivalent} to a free partition quantum group ? We will come back to this problem in Subsection \ref{sec:nonunimodular}. We end this section by giving a set of generators of the category of partitions $\CC_{S}$. For two words $w, w'$ on $S$, let $\pi(w, w')\in NC^{\A_{S}}$ be the one-block partition with upper colouring $w$ and lower colouring $w'$.

\begin{lem}\label{lem:categorygenerators}
Let $S$ be a fusion set and let $\CC'$ be the category of partitions generated by all the partitions $\pi(xy, z)$ for $x, y, z\in S$ satisfying $x\ast y = z$. Then, $\CC' = \CC_{S}$.
\end{lem}

\begin{proof}
Note that $\CC'\subset \CC_{S}$ by definition. Moreover, $\CC'_{S}\subset \CC'$ so that we only have to deal with $\CC_{\HH(S)}$. Assume now that $\pi(w, w')\in \CC_{\HH(S)}$ with $w = w_{1}\dots w_{n}$ and define a sequence $z_{i}$ of elements of $S$ by $z_{1} = w_{1}$ and $z_{i+1} = z_{i}\ast w_{i+1}$. Let us prove by induction that $\pi(w_{1}\dots w_{i}, z_{i})\in \CC'$. For $n=1$, this is clear. For $n > 1$, this comes from the equality
\begin{equation*}
\pi(w_{1}\dots w_{i+1}, z_{i+1}) = \pi(z_{i}w_{i+1}, z_{i+1})(\pi(w_{1}\dots w_{i}, z_{i})\otimes \vert).
\end{equation*}
Since $z_{n} = f(w)$ by construction, we have proved that $\pi(w, f(w))\in \CC'$. The same proof yields $\pi(w', f(w'))\in \CC'$, so that $\pi(w, w')\in \CC'$. Using rotations, we see that any one-block partition of $\CC_{S}$ is in $\CC'$. Because $\CC_{S}$ is block-stable, any partition in it can be built from one-bock partitions using the category operations. Hence, $\CC_{S}\subset \CC'$, concluding the proof.
\end{proof}

\subsection{Classification of fusion sets}\label{sec:classification}

We now want to investigate the general structure of admissible fusion sets and to classify them. The outcome of this study will be a general free product decomposition for free partition quantum groups, given in Theorem \ref{thm:freeproduct}. When the colour set $\A$ contains two colours which are inverse to one another, such a classification is known by \cite[Thm 4.23]{freslon2013fusion}. Since it will serve as a building block for the general case, let us restate this result in the context of fusion sets.

\begin{thm}\label{thm:classificationonegenerator}
Let $\CC$ be a block-stable category of noncrossing partitions with two colours which are conjugate to each other. Then, $S(\CC)$ is isomorphic to one of the following :
\begin{itemize}
\item $\OO = \{x\}$ with $x\ast x = \emptyset$.
\item $\UU = \{x, \co{x}\}$ with $x\ast x = \co{x}\ast x = x\ast\co{x} = \co{x}\ast\co{x} = \emptyset$.
\item $\Z_{s}$ for $1\leqslant s\leqslant \infty$, where the conjugation and fusion operations are given by the group inverse and group operation.
\item $\ZZ = \{x, \co{x}, x\ast\co{x}, \co{x}\ast x\}$ with the fusion operation given by \cite[Def 4.22]{freslon2013fusion} (where it is denoted by $\mathcal{S}$). Note that this is the fusion set associated to a groupoid with two elements having trivial automorphism groups and two mutually inverse morphisms between these elements.
\end{itemize}
\end{thm}

In Theorem \ref{thm:classification}, we shall see that any admissible fusion set can roughly be built from these ones (with $\Z_{s}$ replaced by arbitrary discrete groups). In order to get such a result, we first have to decompose admissible fusion sets in blocks using the following notion of \emph{disjoint union}. From now on, all fusion sets will be assumed to be admissible.

\begin{de}
Let $S, S'$ be two fusion sets. Their \emph{disjoint union} is the set $S'' = S\sqcup S'$ together with its natural conjugation and the fusion operation induced by the fusion operations on $S$ and $S'$ with the following additional rule : $x\ast y = \emptyset$ whenever $x\in S$ and $y\in S$.
\end{de}

We will proceed stepwise to decompose $S$ as a disjoint union of simpler fusion sets. A first decomposition of this kind was already introduced in Subsection \ref{subsec:partitions} to define the category of partitions $\CC_{S}$. We refine it by defining the following three subsets :
\begin{itemize}
\item $\HH(S) = \{x\in S, \co{x}\ast x\neq\emptyset\}$
\item $\OO(S) = \{x\notin \HH(S), x = \co{x}\}$
\item $\UU(S) = \{x\notin \HH(S), x \neq \co{x}\}$
\end{itemize}

\begin{lem}
The restrictions of the conjugation and the fusion operation turn the three sets above into fusion sets. Moreover, $S$ is their disjoint union.
\end{lem}

\begin{proof}
Let us recall that by Lemma \ref{lem:inverse}, if $x, y\in S$ are such that $x\ast y \neq \emptyset$, then $\co{x}\ast x\neq \emptyset$ (and $x\ast \co{x}\neq \emptyset$ too). Thus, for any $x\in \OO(S)$ or $x\in \UU(S)$ and $y\in S$, $x\ast y = \emptyset$. This proves that these two sets are fusion subsets and that they are disjoint from each other and from $\HH(S)$. We therefore only have to prove that $\HH(S)$ is stable under the  fusion operation. This is a consequence of the fact that if $x\ast y = z$, then $\co{z}\ast z = \co{y}\ast\co{x}\ast x\ast y = \co{y}\ast y \neq\emptyset$ (using Lemma \ref{lem:inverse}).
\end{proof}

\begin{prop}\label{prop:firstdecomposition}
Let $S$ be an admissible fusion set. Then, $S$ is a the disjoint union of $\HH(S)$, copies of $\OO$ and copies of $\UU$.
\end{prop}

\begin{proof}
Writing $\OO(S)$ as a disjoint union of singletons gives a decomposition into disjoint fusion sets. By the classification of the two-colour case, these are all isomorphic to $\OO$. Similarly, writing $\UU(S)$ as a disjoint union of pairs $\{x, \co{x}\}$ gives a disjoint union of fusion sets isomorphic to $\UU$. The result therefore follows from the decomposition $S = \OO(S)\sqcup \UU(S)\sqcup \HH(S)$.
\end{proof}

The problem now reduces to describing $\HH(S)$. Again, we define two subsets :
\begin{itemize}
\item $\Gamma(S) = \{x\in S, x\ast \co{x} = \co{x}\ast x \}$.
\item $\ZZ(S) = \{x\in S, x\ast \co{x} \neq \co{x}\ast x \}$.
\end{itemize}

\begin{lem}
The subset $\Gamma(S)$ is stable under the conjugation and the fusion operation. Moreover, there is a family of discrete groups $(\Gamma_{i})_{i}$ such that $\Gamma(S) = \sqcup_{i}\Gamma_{i}$ as fusion sets.
\end{lem}

\begin{proof}
The stability under conjugation is clear. For the fusion operation, note that if $x\ast y = z$, then
\begin{equation*}
\co{x}\ast x = (\co{z\ast\co{y}})\ast (z\ast \co{y}) = y\ast \co{y}.
\end{equation*}
Thus, if $x, y\in \Gamma(S)$ and $x\ast y = z$, then
\begin{equation*}
z\ast \co{z} = x\ast \co{x} = \co{x}\ast x = y\ast \co{y} = \co{y}\ast y = \co{z}\ast z
\end{equation*}
and $\Gamma(S)$ is stable. Now, let us define a binary relation on $\Gamma(S)$ by $x\sim y$ if $x\ast \co{y}\neq \emptyset$. The key fact is that, by Lemma \ref{lem:inverse}, $x\sim y$ if and only if $\co{x}\ast x = \co{y}\ast y$. From this, it is straightforward to see that $\sim$ is an equivalence relation on $\Gamma(S)$. Let $(\Gamma_{i})_{i}$ be the equivalence classes for $\sim$. Then, $\ast$ turns $\Gamma_{i}$ into a monoid. Moreover, $\co{x}\ast x = y\ast \co{y} = \co{y}\ast y$ for any $x, y\in \Gamma_{i}$ so that $\co{x}\ast x$ is a neutral element and $\co{x}$ is an inverse for $x$. Summing up, each $\Gamma_{i}$ is a group, concluding the proof.
\end{proof}

The difficulty is to deal with $\ZZ(S)$. Noting that for any $x\in \ZZ(S)$, $\co{x}\ast x\in \Gamma(S)$, we see that $\ZZ(S)$ is never a fusion subset and that we have to keep the interplay between $\ZZ(S)$ and $\Gamma(S)$ in the picture. We can however isolate the part in $\Gamma(S)$ which does not interact with $\ZZ(S)$ by setting :
\begin{itemize}
\item $\Gamma'(S) = \{x\in \Gamma(S), \forall y\in \ZZ(S), x\ast y = \emptyset\}$.
\item $\Lambda(S) = \Gamma(S)\backslash \Gamma'(S)$ (the complement of $\Gamma'(S)$).
\item $\HH'(S) = \Lambda(S)\cup \ZZ(S)$.
\end{itemize} 

\begin{prop}\label{prop:gammaprime}
The sets $\Gamma'(S)$ and $\HH'(S)$ are disjoint fusion subsets of $\HH(S)$.
\end{prop}

\begin{proof}
By associativity, $(x\ast x')\ast y = \emptyset$ if $x'\ast y = \emptyset$, so that $\Gamma'(S)$ is a fusion subset. Recall that $\Gamma(S) = \sqcup_{i\in I}\Gamma_{i}$ for some discrete groups $\Gamma_{i}$ and that two given elements $x$ and $y$ belong to the same $\Gamma_{i}$ if and only if $\co{x}\ast x = \co{y}\ast y$. This implies that there is a subset $J\subset I$ such that $\Gamma'(S) = \sqcup_{i\notin J}\Gamma_{i}$ and $\Lambda(S) = \sqcup_{i\in J}\Gamma_{i}$. Thus, $\Gamma'(S)$ and $\HH'(S)$ are disjoint.
\end{proof}

From now on, we will assume that $S = \HH'(S)$ and further describe the structure of this fusion set in three steps. We will describe some fusion sets by generators and relations, so that we first have to make this notion clear.

\begin{lem}\label{lem:universalproperty}
Let $\A$ be a colour set and let $B$ be a set of pairs $(w, w')$, where $w$ and $w'$ are words on $\A$. Then, there exists a fusion set $S$ with the following properties :
\begin{enumerate}
\item $S$ is generated by $\A$.
\item For any $(w, w')\in B$, $f(w) = f(w')$ in $S$.
\item For any fusion set $T$ generated by $\A$ and such that for any $(w, w')\in B$, $f(w) = f(w')$ in $T$, there exists a surjective morphism of fusion sets
\begin{equation*}
\Phi : S \longrightarrow T
\end{equation*}
which is the identity on $\A$.
\end{enumerate}
Moreover, such a fusion set $S$ is unique up to isomorphism and called the \emph{fusion set generated by $\A$ and the relations $f(w) = f(w')$} for $(w, w')\in B$.
\end{lem}

\begin{proof}
The uniqueness is clear. To prove the existence, let $\CC$ be the category of partitions generated by the $x$-identity for every $x\in \A$ and the partitions $\pi(w, w')$ for every $(w, w')\in B$. Then, $S = S(\CC)$ satisfies all the required properties.
\end{proof}

\subsubsection{First step}

Recall that for any $x\in \ZZ(S)$, $x\ast\co{x}\in \Lambda(S)$. Let us denote by $\Lambda_{x}$ the group in $\Lambda(S)$ containing this element. Here are alternate characterizations :

\begin{lem}
Let $t\in \Lambda(S)$. Then, the following are equivalent :
\begin{itemize}
\item $t\in \Lambda_{x}$
\item $t\ast \co{t} = x\ast \co{x}$
\item $\co{t}\ast t = x\ast \co{x}$
\item $t\ast x\neq\emptyset$
\end{itemize}
Moreover, the map $f_{x} : t\mapsto \co{x}\ast t\ast x$ induces an isomorphism between $\Lambda_{x}$ and $\Lambda_{\co{x}}$.
\end{lem}

\begin{proof}
An element $t\in \Lambda(S)$ is in $\Lambda_{x}$ if and only if $x\ast \co{x}\sim t$. By definition, this means that $(x\ast\co{x})\ast \co{t}\neq\emptyset$, which is equivalent to $t\ast \co{t} = \co{t}\ast t = (\co{x\ast\co{x}})\ast(x\ast\co{x}) = x\ast \co{x}$. This gives the first three equivalences. The last one follows directly from the fact that $a\ast b\neq\emptyset$ if and only if $\co{a}\ast a = b\ast\co{b}$. If $t\ast \co{t} = x\ast \co{x}$, then $f_{x}(t)\ast \co{f_{x}(t)} = \co{x}\ast x$, so that $f_{x}$ maps $\Lambda_{x}$ to $\Lambda_{\co{x}}$. It is a bijection because $f_{\co{x}}$ is an inverse and a group homomorphism by a straightforward computation using Lemma \ref{lem:inverse}.
\end{proof}

We will now use a second binary relation $\approx$ on $\ZZ(S)$ defined by : $x\approx y$ if $x\ast\co{y}\in \Lambda(S)$.

\begin{lem}
The binary relation $\approx$ is an equivalence relation on $\mathcal{Z}(S)$.
\end{lem}

\begin{proof}
Because $x\ast\co{x}\in \Lambda_{x}$ and $y\ast\co{x} = \co{x\ast \co{y}}$, $\approx$ is reflexive and symmetric. To prove that it is transitive, note that if $x\approx y$ and $y\approx z$, then $x\ast\co{y}\in \Lambda_{x}$ and $y\ast\co{z}\in\Lambda_{y}$, so that
\begin{equation*}
x\ast\co{z} = (x\ast\co{y})\ast(y\ast\co{z})\in \Lambda_{x}.
\end{equation*}
\end{proof}

Let $\mathcal{R}(S)$ be a set of representatives of the equivalence classes for $\approx$ with the property that if $x\in \mathcal{R}(S)$ then $\co{x}\in \mathcal{R}(S)$ (this makes sense because $x\in \ZZ(S)$ can never be equivalent to $\co{x}$). We will denote the equivalence class of $x$ by $[x]$. There is in fact a strong link between $[x]$ and the group $\Lambda_{x}$.

\begin{lem}\label{lem:freeaction}
The fusion operation induces an action of $\Lambda_{x}$ on $[x]$ which is free and transitive. In particular, the map
\begin{equation*}
\begin{array}{ccc}
\Lambda_{x} & \rightarrow & [x] \\
t & \mapsto & t\ast x
\end{array}
\end{equation*}
is a bijection.
\end{lem}

\begin{proof}
If $t\ast y = y$, then $t = y\ast\co{y}$. Since $y\ast\co{y}$ is the neutral element for $\Lambda_{x}$ by Lemma \ref{lem:inverse}, we conclude that $t$ is neutral, i.e. the action is free. Moreover, if $y\in [x]$, then $t_{y} = y\ast \co{x} = \co{x\ast\co{y}}\in \Lambda_{x}$ and $y = t_{y}\ast x$, so that the action is transitive.
\end{proof}

We can now have a first glimpse of the general structure of $\HH'(S)$ by describing its basic "blocks". For $x\in \mathcal{R}(S)$, let us set $\ZZ_{x} = [x]\cup [\co{x}]$. Let us also denote by $e_{\Lambda}$ the neutral element of a group $\Lambda$.

\begin{prop}\label{prop:presentationfusionset}
The set $S_{x} = \ZZ_{x}\cup(\Lambda_{x}\sqcup \Lambda_{\co{x}})$ is a fusion set. Moreover, let $\Lambda$ be a group isomorphic to $\Lambda_{x}$ and let $S'$ be the fusion set generated by $\Lambda$ and an element $a\notin\Lambda$ with the relation $a\ast \co{a} = e_{\Lambda}$. Then, there is an isomorphism of fusion sets $S' \rightarrow S_{x}$ sending $a$ to $x$ and $\Lambda$ to $\Lambda_{x}$.
\end{prop}

\begin{proof}
For the first assertion, the only thing to prove is the stability under the fusion operation. We check it case by case :
\begin{itemize}
\item If $a$ and $b$ belong to the same set of the union, then $a\ast b$ also belongs to this set as soon as it is not $\emptyset$.
\item If $a\in \Lambda_{x}$ and $b\in \Lambda_{\co{x}}\cup[\co{x}]$ then $a\ast b = \emptyset$. Similarly with $x$ replaced by $\co{x}$.
\item If $a\in [x]$ and $b\in [\co{x}]$, then $a\ast b\in \Lambda_{x}$.
\end{itemize}
If $\varphi : \Lambda \rightarrow \Lambda_{x}$ is an isomorphism, there is by Lemma \ref{lem:universalproperty} a surjection of fusion sets $\Phi : S'\rightarrow S_{x}$ which restricts to $\varphi$ on $\Lambda$ and sends $a$ to $x$. Since $\ZZ(S')$ is generated by $a$, it is equal to $[a]\cup[\co{a}] = (\Lambda_{a}\ast a) \cup(\Lambda_{\co{a}}\ast\co{a})$ by Lemma \ref{lem:freeaction}. Therefore, it is enough to prove that $\Lambda_{a} = \Lambda$ to conclude that $\Phi$ is an isomorphism. Let $w$ be a word on $\Lambda\cup\{a, \co{a}\}$ such that $f(w)\in \Lambda_{a}$. We will prove by induction on the length of $w$ that $f(w)\in \Lambda$. If $w$ is of length one or two, this is clear. If $w = w_{1}\dots w_{n}$ with $n\geqslant 2$, we have $f(w)\ast\co{f(w)} = w_{1}\ast\co{w}_{1} = a\ast \co{a}$, so that either $w_{1}\in \Lambda$ or $w_{1} = a$. In the first case, we must also have $f(w_{2}\dots w_{n})\in \Lambda_{a}$. Thus, by induction, $f(w) = w_{1}\ast f(w_{2}\dots w_{n})\in \Lambda$. In the second case, we must have $w_{2} = \co{a}$ ($a\ast t\neq\emptyset$ for some $t\in \Lambda$ would imply $a\in \Lambda$), so that $f(w) = f(w_{3}\dots w_{n})\in \Lambda$ again by induction. Thus, $\Lambda_{a} = \Lambda$ and $\Phi$ is an isomorphism.
\end{proof}

\subsubsection{Second step}

Obviously, $S = \cup_{x\in \mathcal{R}(S)}S_{x}$. However, these fusion sets are far from being disjoint in general. Our second step is to understand their interplay by taking in account the relations $x\ast y \neq\emptyset$ for $x, y\in \mathcal{R}(S)$, which are equivalent to $\Lambda_{y} = \Lambda_{\co{x}}$. To do this, recall that $\Lambda(S) = \sqcup_{j\in J}\Lambda_{j}$ and set, for $j\in J$,
\begin{equation*}
\mathcal{R}_{j} = \{x\in \mathcal{R}(S), \Lambda_{x} = \Lambda_{j}\}.
\end{equation*}

\begin{lem}\label{lem:decompositionsecondstep}
Let $S_{j}$ be the fusion subset of $S$ generated by $\mathcal{R}_{j}$. Then,
\begin{equation*}
S_{j} = \left(\bigcup_{x\in \mathcal{R}_{j}} S_{x}\right) \cup \left(\bigcup_{x, y\in \mathcal{R}_{j}} S_{\co{x}\ast y}\right).
\end{equation*}
\end{lem}

\begin{proof}
Let us denote by $T$ the right-hand side of the equality. Since $S_{x}$ is the fusion subset generated by $x$, it is clear that $T\subset S_{j}$. Thus, the only thing we have to prove is that $T$ is a fusion set. The stability under the conjugation operation is clear. For the fusion operation, consider the following facts :
\begin{itemize}
\item If $x, y\in \mathcal{R}_{j}$, then $x\ast y = \emptyset$.
\item If $x, y\in \mathcal{R}_{j}$, then $\co{x}\ast y\in S_{\co{x}\ast y}$.
\item If $x, y, y'\in \mathcal{R}_{j}$, then $(\co{x}\ast y)\ast \co{y}'\neq \emptyset$ implies $\Lambda_{\co{y}} = \Lambda_{\co{y}'}$. But since $\Lambda_{y} = \Lambda_{y'}$ by definition, we get $y \approx y'$. Because we are considering representatives of the equivalence classes, this implies $y = y'$. Thus, the product lies in $S_{x}$. A similar argument works for $(\co{x}\ast y)\ast (\co{x}'\ast y')$ and $x\ast(\co{x}'\ast y)$.
\end{itemize}
By the description of $S_{x}$ given in Proposition \ref{prop:presentationfusionset}, this proves that $T$ is stable under the fusion operation. Hence, it is a fusion set and $T = S_{j}$.
\end{proof}

As for the first step, we want to give an abstract presentation of the fusion set $S_{j}$.

\begin{prop}\label{prop:presentationfusionsetbis}
Let $\Lambda$ be a group isomorphic to $\Lambda_{j}$ and let $M$ be a set in bijection with $\mathcal{R}_{j}$. Let $S'$ be the fusion set generated by $\Lambda$ and the elements of $M$ with the relations $b\ast\co{b} = e_{\Lambda}$ for every $b\in M$. Then, $S' = S_{j}$.
\end{prop}

\begin{proof}
Let $\varphi : \Lambda \rightarrow \Lambda_{j}$ be an isomorphism and let $g : M\rightarrow \mathcal{R}_{j}$ be a bijection. By construction, there is a surjective morphism of fusion sets $\Phi : S' \rightarrow S_{j}$ which restricts to $\varphi$ on $\Lambda$ and sends $b\in M$ to $g(b)$. Let $S'_{b}$ be the fusion subset of $S'$ generated by $b$ and $\Lambda$. By Proposition \ref{prop:presentationfusionset}, $\Phi$ restricts to an isomorphism between $S'_{b}$ and $S_{g(b)}$. The same holds for $S'_{\co{b}_{1}\ast b_{2}}$. Now, let $a_{1}, a_{2}\in S'$ be such that $\Phi(a_{1}) = \Phi(a_{2})$. There exists $\lambda_{1}, \lambda_{2}\in \Lambda$ and $b_{1}, b_{2}\in M$ such that $a_{i} = \varphi(\lambda_{i})\ast g(b_{i})$ for $i=1, 2$. The equality $\Phi(a_{1}) = \Phi(a_{2})$ then yields
\begin{equation*}
g(b_{1}) = \overline{\varphi(\lambda_{1})}\ast\varphi(\lambda_{1})\ast g(b_{1}) = \overline{\varphi(\lambda_{1})}\ast\varphi(\lambda_{2})\ast g(b_{2}) = \varphi(\lambda_{1}^{-1}\lambda_{2})\ast g(b_{2}).
\end{equation*}
This means that $g(b_{1}) \approx g(b_{2})$, which implies $g(b_{1}) = g(b_{2})$. By injectivity of $g$, we have $b_{1} = b_{2}$. Thus, $\varphi(\lambda_{1}^{-1}\lambda_{2}) = e_{\Lambda_{j}}$ and by injectivity of $\varphi$, $\lambda_{1} = \lambda_{2}$, concluding the proof.
\end{proof}

Using this picture, we can give an alternative characterization of $\HH'(S)$ using \emph{groupoids} as in Example \ref{ex:groupoid}. In fact, consider the small category $\Gr(S)$ with $J$ as set of objects and with morphism spaces :
\begin{equation*}
\Mor_{\Gr(S)}(i, j) = \{x\in \HH'(S), \Lambda_{\co{x}} = \Lambda_{j}, \Lambda_{x} = \Lambda_{i}\}.
\end{equation*}
The fusion operation induces an associative composition for which $\co{x}$ is the inverse of $x$. Thus, $\Gr(S)$ is a groupoid. Note that $\Mor(i, i)$ is precisely the group $\Lambda_{i}$.

\begin{rem}
For $x\in \Mor_{\Gr(S)}(i, j)$, the source of $x$ should act on the right, hence the reversal of notations in the definition if we want the source to be $\Lambda_{i}$.
\end{rem}

We can summarize our results so far as follows :

\begin{cor}\label{cor:groupoiddecomposition}
Any admissible fusion set is a disjoint union of copies of $\OO$, copies of $\UU$, groups and $S_{\Gr}$ for some groupoid $\Gr$.
\end{cor}

\subsubsection{Third step}

Again, we have by construction $S = \cup_{j\in J} S_{j}$ and we should investigate the intersections of these sets. However, it will appear that they are either identical or disjoint, so that the decomposition will be complete. We start with a natural notion of connectedness.

\begin{de}
An admissible fusion set $S$ is said to be \emph{connected} if for any $i, j\in J$, there is an element $y\in \ZZ(S)$ such that $\Lambda_{y} = \Lambda_{i}$, $\Lambda_{\co{y}} = \Lambda_{j}$.
\end{de}

\begin{lem}\label{lem:connectedcomponentsdisjoint}
The fusion set $\HH'(S)$ is a disjoint union of connected fusion sets.
\end{lem}

\begin{proof}
Let us define a binary relation $\backsim$ on the set $J$ by : $i\backsim j$ if there exists $y\in \ZZ(S)$ such that $\Lambda_{y} = i$ and $\Lambda_{\co{y}} = j$. This is clearly an equivalence relation. Let $K$ be a set of representatives for $\backsim$ and set, for $k\in K$,
\begin{equation*}
T_{k} = \{x\in S, \exists j\backsim k, \Lambda_{x} = \Lambda_{j}\}.
\end{equation*}
Then $\HH'(S) = \cup_{k}T_{k}$. Moreover, each $T_{k}$ is a connected fusion set by construction. Let now $x\in T_{k}$ and $x'\in T_{k'}$ be such that $x\ast x'\neq \emptyset$. Then,
\begin{equation*}
\Lambda_{k'} = \Lambda_{x'} = \Lambda_{\co{x}} = \Lambda_{j}.
\end{equation*}
By definition of $\backsim$, $j\backsim k$ so that in the end, $k = k'$. Hence, $\HH'(S)$ is the disjoint union of the fusion sets $T_{k}$.
\end{proof}

To conclude, we now only have to see that the fusion set $T_{k}$ is in fact equal to $S_{k}$.

\begin{lem}\label{lem:connectedpresentation}
The fusion sets $S_{j}$ are connected. Moreover, if $i, j\in J$ satisfy $i\backsim j$, then $S_{i} = S_{j}$. Hence, $S_{i} = T_{i}$.
\end{lem}

\begin{proof}
The fusion set $S_{j}$ is connected because it is generate by elements satisfying $\Lambda_{x} = \Lambda_{j}$. Assume now that $i\backsim j$. Let $x\in \mathcal{R}_{j}$ and let $y\in \ZZ(S)$ be such that $\Lambda_{y} = \Lambda_{i}$ and $\Lambda_{\co{y}} = \Lambda_{j}$. Then, $\Lambda_{y\ast x} = \Lambda_{i}$ so that $y\ast x \in S_{i}$. This implies that $x = (\co{y\ast x})\ast y \in S_{i}$, thus $S_{j} \subset S_{i}$. By symmetry of $\backsim$, $S_{i} = S_{j}$ and the proof is complete.
\end{proof}

We can now give a presentation of any admissible fusion set. For clarity, let us first give a definition.

\begin{de}
A \emph{fusion triple} $\Theta = (n_{\OO}, n_{\UU}, (\Gamma_{i}, n_{i})_{i\in I})$ is given by :
\begin{itemize}
\item An integer $n_{\OO}$
\item An integer $n_{\UU}$
\item A collection of groups $(\Gamma_{i})_{i\in I}$
\item For every $i\in I$, an integer $n_{i}$
\end{itemize}
\end{de}

Let $\Theta$ be a fusion triple and let $J$ be the set of elements $j\in I$ such that $n_{j} > 0$. For $j\in J$, let $S_{j}$ be the fusion set generated by $\Gamma_{j}$ and $n_{j}$ elements $b_{1}^{j}, \dots, b_{n_{j}}^{j}$ with the relations $b_{k}^{j}\ast\co{b}_{k}^{j} = e_{\Gamma_{j}}$. The fusion set $S(\Theta)$ associated to the fusion triple $\Theta$ is the disjoint union of $n_{\OO}$ copies of $\OO$, $n_{\UU}$ copies of $\UU$, the fusion set $\sqcup_{i\notin J} \Gamma_{i}$ and the fusion sets $S_{j}$.

\begin{thm}\label{thm:classification}
There is a one-to-one correspondence between isomorphism classes of admissible fusion sets and isomorphism classes (in an obvious sense) of fusion triples.
\end{thm}

\begin{proof}
Let $S$ be a fusion set. Let $n_{\OO}(S)$ and $n_{\UU}(S)$ be the number of copies of $\OO$ and $\UU$ appearing in Proposition \ref{prop:firstdecomposition} and set $\Gamma(S) = \sqcup_{i}\Gamma_{i}$ as in Proposition \ref{prop:gammaprime}. Combining Lemma \ref{lem:connectedcomponentsdisjoint} and Lemma \ref{lem:connectedpresentation}, we see that $\HH'(S) = \sqcup_{k\in K} S_{k}$. To each $S_{k}$, associate the group $\Lambda_{k}$ and the integer $n_{k} = \vert \mathcal{R}_{k}\vert - 1$. This forms a fusion triple $\Theta(S)$. The result then follows from the fact that $S(\Theta(S))$ is isomorphic to $S$ by Proposition \ref{prop:presentationfusionsetbis} and Lemma \ref{lem:connectedpresentation} and the fact that $\Theta(S(\Theta))$ is by construction isomorphic to $\Theta$.
\end{proof}

\subsection{Free product decomposition}

Going back to compact quantum groups, we would like to give an interpretation of Theorem \ref{thm:classification}. This will lead to a general decomposition result for free partition quantum groups into free products of elementary ones. Let us first see how the fact that fusion sets are disjoint translates at the level of quantum groups.

\begin{lem}\label{lem:freeproduct}
Let $N\geqslant 4$ be an integer, let $S$ and $S'$ be fusion sets and let $\G$ and $\G'$ be the associated partition quantum groups. Then, the partition quantum group $\G''$ associated to the disjoint union $S'' = S\sqcup S'$ is isomorphic to the free product $\G\ast \G'$.
\end{lem}

\begin{proof}
By definition of the disjoint union we have, at the level of categories of partitions, a free product decomposition $\CC_{S''} = \CC_{S}\ast \CC_{S'}$. We can therefore conclude by Proposition \ref{prop:freeproduct}.
\end{proof}

As an immediate consequence, we get our first decomposition.

\begin{prop}\label{prop:firstfreeproddecomposition}
Let $S$ be an admissible fusion set and let $\G'$ be the partition quantum group associated to $\HH'(S)$. Then, the partition quantum group associated to $S$ is the free product of $\G'$, copies of $O_{N}^{+}$, copies of $U_{N}^{+}$ and free wreath products of discrete groups by $S_{N}^{+}$.
\end{prop}

\begin{proof}
The free product decomposition comes from Proposition \ref{prop:firstdecomposition}, Proposition \ref{prop:gammaprime} and Lemma \ref{lem:freeproduct}. It is then clear on the construction that the partition quantum group associated to $\OO$, $\UU$ and a group $\Gamma$ respectively  are $O_{N}^{+}$, $U_{N}^{+}$ and the free wreath product $\Gamma\wr_{\ast}S_{N}^{+}$, hence the result.
\end{proof}

Combining Proposition \ref{prop:presentationfusionsetbis} and Lemma \ref{lem:freeproduct}, we see that the partition quantum group associated to $\HH'(S)$ is the free product of the partition quantum groups associated to $S_{k}$ for $k\in K$. These are not free products of unitary easy quantum groups, but can be easily described using amalgamation over $S_{N}^{+}$. To see this, let us first recall from \cite[Prop 3.2]{lemeux2013fusion} that the coefficients of $u^{e_{\Lambda}}$ generate a C*-subalgebra of $C(\h{\Lambda}\wr_{\ast}S_{N}^{+})$ which is isomorphic to $C(S_{N}^{+})$. Moreover, the restriction of the coproduct yields the quantum group structure of $S_{N}^{+}$. This will be identified with a copy of $S_{N}^{+}$ inside $\widetilde{H}_{N}^{+}$ which we now describe.

\begin{lem}
Let $\ZZ = \{x, \co{x}, x\ast \co{x}, \co{x}\ast x\}$ be the fusion set of the free quantum group $\widetilde{H}_{N}^{+}$. Let $A$ be the C*-subalgebra of $C(\widetilde{H}_{N}^{+})$ generated by the coefficients of the representation $u^{x\ast \co{x}}$ and let $B$ be the one generated by the coefficients of $u^{\co{x}\ast x}$. Then, $A$ and $B$ are isomorphic to $C(S_{N}^{+})$. Moreover, the restriction of the coproduct to these algebras yields the coproduct of the corresponding quantum groups.
\end{lem}

\begin{proof}
Let $\pi : C(\widetilde{H}_{N}^{+}) \rightarrow C(S_{N}^{+})$ be the canonical surjection. It sends $u^{x}\otimes u^{\co{x}}$ to $v\otimes v$, where $v$ is the fundamental representation of $S_{N}^{+}$, and it is easy to see that the subrepresentations $u^{x\ast \co{x}}$ and $u^{\co{x}\ast x}$ are sent to the unique copy of $v$ inside $v\otimes v$. From this and the fusion rules, it is straightforward that $\pi$ restricts to isomorphisms on both $A$ and $B$. Since $\pi$ commutes with the coproducts, $(A, \D)$ and $(B, \D)$ are both isomorphic to $S_{N}^{+}$.
\end{proof}

\begin{rem}
One can even prove that the C*-subalgebras $A$ and $B$ are free in $C(\widetilde{H}_{N}^{+})$ with respect to the Haar state. In fact, they generate a copy of $C(S_{N}^{+}\ast S_{N}^{+})$.
\end{rem}

For a group $\Lambda$ and an integer $n$, let us consider a family $(\G_{i})_{1\leqslant i\leqslant n}$ of copies of $\widetilde{H}_{N}^{+}$ and define
\begin{equation*}
Z_{N}^{+}(\Lambda, n) = (\ast_{S_{N}^{+}} \G_{i}) \underset{S_{N}^{+}}{\ast} \left(\h{\Lambda}\wr_{\ast}S_{N}^{+}\right)
\end{equation*}
Note that $Z_{N}^{+}(\Lambda, 0) = \h{\Lambda}\wr_{\ast}S_{N}^{+}$ and that $Z_{N}^{+}(\{e\}, 1) = \widetilde{H}_{N}^{+}$. Our decomposition will rely on an elementary computation.

\begin{lem}\label{lem:amalgamation}
Let $N\geqslant 4$ be an integer, let $\CC$ be a category of partitions, let $\G$ be the associated partition quantum group and let $x, y, z\in \A$. If $\pi(xy, z)\in \CC$, then $u^{x}_{ik}u^{y}_{jl} = \delta_{i, j}\delta_{k, l} u^{z}_{ik}$ for all $1\leqslant i, j, k, l\leqslant N$.
\end{lem}

\begin{proof}
Set $T = T_{\pi(xy, z)}$. We have
\begin{equation*}
u^{z}\circ(\Id\otimes T)(e^{x}_{i}\otimes e^{y}_{j}) = \delta_{i, j}\sum_{k} u^{z}_{ik}\otimes e^{z}_{k}
\end{equation*}
and
\begin{equation*}
(\Id\otimes T)\circ (u^{x}\otimes u^{y})(e^{x}_{i}\otimes e^{y}_{j}) = \sum_{k}u^{x}_{ik}u^{y}_{jk}\otimes e^{z}_{k}.
\end{equation*}
If $\pi(xy, z)\in \CC$, then $T$ is an intertwiner and the two equations above must yield the same result. Because $\pi(xy, xy) = \pi(xy, z)^{*}\pi(xy, z)\in \CC$, a similar computation shows that $u^{x}_{ik}u^{y}_{jl} = 0$ as soon as $i\neq j$ or $k\neq l$. Combining the two equalities yields the statement.
\end{proof}

\begin{lem}\label{lem:zn}
The quantum group $Z_{N}^{+}(\Lambda, n)$ is a partition quantum group. Moreover, the fusion set of $Z_{N}^{+}(\Lambda_{j}, n_{j})$ is isomorphic to $S_{j}$.
\end{lem}

\begin{proof}
For $1\leqslant i\leqslant n$, let $\CC_{i}$ be a copy of the category of partitions $\CC^{\circ, \bullet}_{alt}$ corresponding to the quantum group $\widetilde{H}_{N}^{+}$ and consider the category of partitions $\CC'$ obtained by taking the free product of $\CC_{\Lambda}$ and $\CC_{i}$ for $1\leqslant i\leqslant n$. The associated partition quantum group is, by Lemma \ref{lem:freeproduct},
\begin{equation*}
\G' = (\ast_{i=1}^{n}\G_{i}) \ast \left(\h{\Lambda}\wr_{\ast}S_{N}^{+}\right).
\end{equation*}
Let $\CC$ be the category of partitions generated by $\CC'$ and the partitions $\pi(x_{i}\co{x}_{i}, e_{\Lambda})$ for $1\leqslant i\leqslant n$. By Lemma \ref{lem:amalgamation}, passing from $\CC'$ to $\CC$ amounts to quotienting $C(\G')$ by the relations
\begin{equation*}
(u^{x\ast \co{x}}\oplus\varepsilon)_{ij} = (u^{e}\oplus\varepsilon)_{ij}
\end{equation*}
for all $i, j$, where $\varepsilon$ denotes the trivial representation. This is precisely the amalgamation defining $Z_{N}^{+}(\Lambda, n)$, hence the first assertion. The fusion set $S$ associated to $Z_{N}^{+}(\Lambda_{j}, n_{j})$ is by construction generated by $\Lambda_{j}$ and elements $(x_{i})_{1\leqslant i\leqslant n_{j}}$ with the relations $x_{i}\ast\co{x}_{i}\in \Lambda_{j}$. Thus, by Proposition \ref{prop:presentationfusionsetbis}, it is $S_{j}$.
\end{proof}

We are now ready to give a general structure result. For an integer $N\geqslant 4$, let $\mathcal{F}_{N}$ be the collection of partition quantum groups consisting in $O_{N}^{+}$, $U_{N}^{+}$ and $Z_{N}^{+}(\Lambda, n)$ for all groups $\Lambda$ and all integers $n$.

\begin{thm}\label{thm:freeproduct}
Let $S$ be an admissible fusion set and let $N\geqslant 4$ be an integer. Then, the partition quantum group associated to $S$ and $N$ is a free product of elements of $\mathcal{F}_{N}$.
\end{thm}

\begin{proof}
This is the combination of Proposition \ref{prop:firstfreeproddecomposition}, Lemma \ref{lem:zn} and Theorem \ref{thm:classification}.
\end{proof}

\begin{rem}\label{rem:haagerupbis}
Using this free product decomposition and the arguments of \cite[Sec 6]{freslon2013fusion}, Remark \ref{rem:haagerup} yields the following statement : let $S$ be a fusion set such that $\Gamma(S)$ is a disjoint union of finite groups, let $N\geqslant 4$ be an integer and let $\G$ be the associated partition quantum group, then $\G$ has the Haagerup property.
\end{rem}

\begin{rem}
It is tempting to conjecture that if $\CC$ is a block-stable category of noncrossing partitions and $N\geqslant 4$ is an integer, then the associated partition quantum group is a free product of elements of $\mathcal{F}_{N}$. However, this is false because $B_{N}^{+}\notin \mathcal{F}_{N}$. 
\end{rem}

It is useful to restate Theorem \ref{thm:freeproduct} in the language of $R^{+}$-deformations, in particular for further applications to monoidal equivalence in Subsection \ref{sec:nonunimodular}.

\begin{cor}
Let $\G$ be a compact quantum group. Then, $\G$ is free if and only if $\G$ is a $R^{+}$-deformation of a free product of elements of $\mathcal{F}_{N}$ for some $N\geqslant 4$.
\end{cor}

\section{Partitions and tensor categories}\label{sec:tensor}

The aim of this final section is to take a step towards generality by considering more general C*-tensor categories built from partitions. In fact, Definition \ref{de:tensorcategory} defines \emph{concrete} C*-tensor categories in the sense that there is a finite-dimensional Hilbert space attached to each object in such a way that morphisms are linear maps between these spaces. Forgetting about this concrete realization leads to the notion of \emph{abstract} C*-tensor category. This generalization has two applications. The first one is the construction of non-unimodular versions of partition quantum groups, which we define and try to classify in Subsection \ref{sec:nonunimodular}. The second one is the construction of categories of representations for arbitrary (free) wreath product, leading to a more general notion of "decorated" partitions where instead of colours we use a whole C*-tensor category to label points and blocks in the partitions. This is explained in Subsection \ref{subsec:wreath}.

\subsection{Non-unimodularity}\label{sec:nonunimodular}

In this subsection we want to investigate non-unimo\-dular versions of partitions quantum groups. A motivation for this is to refine Theorem \ref{thm:freeproduct}, trying to identify all free quantum groups up to isomorphism rather than simply classify them up to $R^{+}$-deformations. As a typical example, let us consider the free unitary quantum groups $U_{F}^{+}$ for invertible matrices $F$ defined in \cite{van1996universal}. It is known by \cite{banica1997groupe} that they all have a free fusion semiring. However, few of them are isomorphic to $U_{N}^{+}$ for some integer $N$, so that we need a broader setting to encompass them. The key observation is that we have so far always been considering \emph{unimodular} quantum groups, when $U_{F}^{+}$ seldom is.

\begin{de}
A compact quantum group $\G$ is said to be \emph{unimodular} if the contragredient of any unitary representation of $\G$ is again a unitary representation.
\end{de}

What forces partition quantum groups to be unimodular is the definition of the maps $T_{p}$. We shall therefore look for other ways of associating operators to partitions in a coherent way to build a concrete C*-tensor category with duals. Let us first define the abstract categories we will be working with.

\begin{de}
Let $\CC$ be a category of $\A$-coloured partitions and let $\theta\in \C$. The \emph{partition C*-tensor category} associated to $\CC$ with parameter $\theta$ is the category $\mathfrak{C}(\CC, \theta)$ whose objects are words on $\A$ and whose morphisms are
\begin{equation*}
\Mor(w, w') = \Span\{p, p\in \CC(w, w')\}
\end{equation*}
with the composition given by
\begin{equation*}
q\circ p = \theta^{\rl(q, p)}qp.
\end{equation*}
Moreover, $\mathfrak{C}(\CC, \theta)$ is endowed with the involution $p\mapsto p^{*}$ and the tensor product $\otimes$, turning it into a C*-tensor category with duals.
\end{de}

\begin{rem}
One may prefer to work with pseudo-abelian categories and consider the pseudo-abelian completion (or Karoubi envelope) of $\mathfrak{C}(\CC, \theta)$ to be the important object. However, since unitary fiber functors can be extended to this completion, we will restrict ourselves to $\mathfrak{C}(\CC, \theta)$ in the sequel.
\end{rem}

An important family of examples of partition C*-tensor categories was introduced and studied by G. Lehrer and R. Zhang in \cite{lehrer2012brauer} under the name of \emph{Brauer categories}. These are in fact the partition categories associated to the category $P_{2}$ of all pair partitions. Another example of this construction was studied by P. Deligne in \cite{deligne2007categorie}. He considered the category $P$ of all partitions, thus obtaining categories "interpolating" the categories of representations of the symmetric groups $S_{N}$. We will come back to generalizations of this construction in Subsection \ref{subsec:wreath}. To apply Tannaka-Krein duality to such a category and obtain a compact quantum group, we first have to make it concrete. This means that we must associate to every object of the category a Hilbert space, in a way compatible with the tensor structure. Such an association is called a \emph{unitary fiber functor}.

\begin{de}
Let $\mathfrak{C}$ be a C*-tensor category with duals. A \emph{unitary fiber functor} is given by a finite-dimensional Hilbert space $H_{\varphi(x)}$ for every object $x$ of $\mathfrak{C}$ and linear maps
\begin{equation*}
\varphi : \Mor(x_{1}\otimes\dots\otimes x_{n}, y_{1}\otimes\dots\otimes  y_{k}) \rightarrow \B(H_{\varphi(x_{1})}\otimes\dots\otimes H_{\varphi(x_{n})}, H_{\varphi(y_{1})}\otimes\dots\otimes H_{\varphi(y_{k})})
\end{equation*}
satisfying
\begin{eqnarray*}
\varphi(1) & = & 1 \\
\varphi(S^{*}) & = & \varphi(S)^{*} \\
\varphi(S\otimes T) & = & \varphi(S)\otimes \varphi(T) \\
\varphi(S\circ T) & = & \varphi(S)\circ \varphi(T)
\end{eqnarray*}
\end{de}

\begin{rem}
Fiber functors on $\mathfrak{C}(P, \theta)$ are not studied in \cite{deligne2007categorie}, whereas in the case of Brauer categories, two unitary fiber functors are built in \cite{lehrer2012brauer}. One on $\mathfrak{C}(P_{2}, N)$, yielding the orthogonal group $O_{N}$ and one on $\mathfrak{C}(P_{2}, -N)$, yielding the symplectic group $Sp_{N}$. In general, there seems to be very few unitary fiber functors on the representation category of classical groups (see the comment on $SU(2)$ at the end of the introduction of \cite{bichon2006ergodic}).
\end{rem}

If $V$ is a finite-dimensional Hilbert space and if $\CC$ is a category of partitions, then there is a unique unitary fiber functor on $\mathfrak{C}(\CC, \dim(V))$ sending a word $w$ to $V^{\otimes w}$ and a partition $p$ to the linear map $T_{p}$. Thus, partition quantum groups can be seen as unitary fiber functors on partition C*-tensor categories. This suggests the following extension of the definition :

\begin{de}\label{de:partitionquantumgroup}
A \emph{generalized partition quantum group} is the compact quantum group associated to a unitary fiber functor on a partition C*-tensor category. If the unitary fiber functor is faithful, then $\G$ is said to be a \emph{faithful} generalized partition quantum group.
\end{de}

Assume that $\G_{1}$ and $\G_{2}$ are associated to \emph{faithful} unitary fiber functors on the same category. Then, they have the same fusion rules, so that they are $R^{+}$-deformations of one another. Since we know by Proposition \ref{prop:linearindependence} that $p\mapsto T_{p}$ is faithful when $\CC$ is noncrossing, this suggests the following conjecture, for which we will give evidence in the remainder of this section :

\begin{conj}\label{conj:completeclassificationfree}
Let $\G$ be a free compact quantum group. Then, $\G$ is a faithful generalized noncrossing partition quantum group in the sense of Definition \ref{de:partitionquantumgroup}.
\end{conj}

To tackle this conjecture, one has to understand the link between quantum groups coming from faithful unitary fiber functors on the same category. This is captured by the notion of monoidal equivalence, introduced in the setting of compact quantum groups by J. Bichon, A. de Rijdt and S. Vaes in \cite{bichon2006ergodic}.

\begin{de}
Two compact quantum groups $\G_{1}$ and $\G_{2}$ are \emph{monoidally equivalent} if they come from two faithful unitary fiber functors on the same C*-tensor category with duals.
\end{de}

It is clear that being isomorphic is stronger than being monoidally equivalent, which is in turn stronger than being a $R^{+}$-deformation. Conjecture \ref{conj:completeclassificationfree} is in a sense of converse of this for free quantum groups. Here are known results in that direction. The first two assertions come from \cite[Thm 2]{banica1996theorie}, \cite[Thm 2]{banica1997groupe} and \cite[Thm 5.3 and Thm 6.2]{bichon2006ergodic} and the third one is \cite[Thm 1.1]{mrozinski2013quantum} :
\begin{itemize}
\item If $\G$ is a $R^{+}$-deformation of $O_{N}^{+}$, then it is monoidally equivalent to $SU_{q}(2)$ for some $q$ and isomorphic to $O_{F}^{+}$ for some invertible matrix $F$ (see \cite{wang1995free} and \cite{van1996universal} for the definition).
\item If $\G$ is a $R^{+}$-deformation of $U_{N}^{+}$, then it is monoidally equivalent to the free complexification of $SU_{q}(2)$ for some $q$ and isomorphic to $U_{F}^{+}$ for some invertible matrix $F$ (see \cite{wang1995free} and \cite{van1996universal} for the definition).
\item If $\G$ is a $R^{+}$-deformation of $S_{N}^{+}$, then it is monoidally equivalent to $SO_{q}(3)$ for some $q$ and isomorphic to the quantum automorphism group $\Aut(B, \psi)$ of some finite-dimensional C*-algebra $B$ with a $\delta$-form $\psi$ (see \cite{wang1998quantum} and \cite{banica1999symmetries} for the definition).
\end{itemize}

Considering the class of quantum groups $\mathcal{F}_{N}$ of Theorem \ref{thm:freeproduct}, we see that the only basic case missing is $Z_{N}^{+}(\Lambda, n)$. It is reasonable (in particular in view of \cite{pittau2014}) to make the following conjecture, which we split into three cases :

\begin{conj}\label{conj:monoidalequivalence}
Let $\G$ be a compact quantum group.
\begin{itemize}
\item If $\G$ is a $R^{+}$-deformation of $\h{\Gamma}\wr_{\ast} S_{N}^{+}$, then it is monoidally equivalent to $\h{\Gamma}\wr_{\ast}SO_{q}(3)$ for some $q$ and isomorphic to $\h{\Gamma}\wr_{\ast}\Aut(B, \psi)$ for some finite-dimensional C*-algebra $B$ with a $\delta$-form $\psi$.
\item If $\G$ is a $R^{+}$-deformation of $\widetilde{H}_{N}^{+}$, then it is monoidally equivalent to the free complexification of $\h{\Z}_{2}\wr_{\ast}SO_{q}(3)$ for some $q$ and isomorphic to the free complexification $\widetilde{H}^{+}(B, \psi)$ of $\h{\Z}_{2}\wr_{\ast}\Aut(B, \psi)$ for some finite-dimensional C*-algebra $B$ with a $\delta$-form $\psi$.
\item If $\G$ is a $R^{+}$-deformation of $Z_{N}^{+}(\Lambda, n)$, then it is isomorphic to
\begin{equation*}
Z^{+}(\Lambda, B, \psi, n) = (\ast_{\Aut(B, \psi)}^{n}\G_{i})\underset{\Aut(B, \psi)}{\ast}\left(\h{\Lambda}\wr_{\ast}\Aut(B, \psi) \right),
\end{equation*}
where $\G_{i} = \widetilde{H}^{+}(B, \psi)$.
\end{itemize}
\end{conj}

\begin{rem}
Note that this conjecture can be restated in the following way : for these particular categories of partitions, any faithful unitary fiber functor on $\mathfrak{C}(\CC, \theta)$ factorizes through $\mathfrak{C}(NC, \theta)$. This is false for $\mathfrak{C}(NC_{2}, \theta)$ (see \cite[Thm 5.5]{bichon2006ergodic}).
\end{rem}

If this holds, then we can give a positive answer to Conjecture \ref{conj:completeclassificationfree}.

\begin{cor}\label{cor:completeclassification}
Let $\mathcal{F}$ be the class of compact quantum groups containing $O_{F}^{+}$, $U_{F}^{+}$ and $Z^{+}(\Lambda, B, \psi, n)$. If Conjecture \ref{conj:monoidalequivalence} holds, then a compact quantum group $\G$ is free if and only if it is a free product of elements of $\mathcal{F}$.
\end{cor}

\begin{proof}
Assume that $\G$ has free fusion semiring and let $\G_{1}, \dots, \G_{k}$ be elements of $\mathcal{F}_{N}$ such that $\G$ is a $R^{+}$-deformation of $\ast_{i} \G_{i}$. If $u_{i}$ denotes the fundamental representation of $\G_{i}$, it corresponds to a representation of $\G$ generating a compact quantum subgroup $\G'_{i}$ with the same fusion rules as $\G_{i}$. Thus, $\G'_{i}\in \mathcal{F}$. Since the subalgebra generated by all the $C(\G'_{i})$'s  is $C(\G)$, we can conclude that there is a surjective $*$-homomorphism
\begin{equation*}
\Phi : C(\ast_{i}\G'_{i}) \rightarrow C(\G)
\end{equation*}
intertwining the coproducts. This map sends the fundamental representation onto the fundamental representation and yields a dimension-preserving isomorphism at the level of the fusion rules, so that by \cite[Lem 5.3]{banica1999representations} it is a $*$-isomorphism.
\end{proof}

\begin{rem}\label{rem:uniqueness}
Such a free product decomposition, together with the results of \cite{banica1997groupe}, \cite{vaes2007boundary}, \cite{lemeux2013fusion} and \cite{pittau2014}, could give a way to extend results on simplicity and uniqueness of KMS-state on the reduced C*-algebra, or on factoriality and fullness for the von Neumann algebra. However, A. Chirvasitu proves many of these results in the Appendix without resorting to the free product decomposition.
\end{rem}

\begin{rem}
The result of Remark \ref{rem:haagerupbis} could then be extended, yielding : $\G$ has the central Haagerup property (see the end of \cite[Sec 6]{freslon2012examples} for a definition) as soon as $\Gamma(S)$ is a disjoint union of finite groups. This follows from the free product decomposition and the fact that, combining \cite{freslon2013fusion} and \cite{freslon2013ccap}, all elements of $\mathcal{F}$ have the central Haagerup property if $\Gamma(S)$ is finite.
\end{rem}

Using Proposition \ref{prop:abelianization}, Corollary \ref{cor:completeclassification} would also give a partial classification of partition groups, i.e. groups whose Schur-Weyl duality can be described by partitions. More precisely, compacts groups associated to \emph{block-stable} categories of partitions will be direct products of abelianizations of elements of $\mathcal{F}$. Noting that a classical group is always unimodular, this is the same as all direct products of elements of the class $\mathcal{F}_{ab}$ of compact groups containing $O_{N}$, $U_{N}$ and $K\wr S_{N}$ for all abelian compact groups $K$.

\subsection{Wreath products and decorated partitions}\label{subsec:wreath}

We now turn to the second generalization, which is inspired by a recent work of F. Lemeux and P. Tarrago \cite{lemeux2014free}. In this paper, their study of free wreath products of arbitrary compact quantum groups by quantum permutation groups suggests that partition quantum groups are not the most general "combinatorial" structure based on partitions to describe representations of "free" quantum groups. The idea is that instead of using colours corresponding to elements of a discrete group $\Gamma$, one can use colours corresponding to representations of a compact quantum group $\G$. But then, we have to "decorate" the blocks of the partitions with morphisms between tensor products of the corresponding representations. Let us give a formal definition for that.

\begin{de}
Let $\mathfrak{A}$ be a C*-tensor category. An $\mathfrak{A}$-\emph{decorated partition} is a partition $p$ together with :
\begin{itemize}
\item An object of $\mathfrak{A}$ attached to each point of $p$.
\item A morphism of $\mathfrak{A}$ attached to each block of $p$.
\end{itemize}
Moreover, if $b$ is a block of $p$ with upper colouring $x_{1}, \dots, x_{n}$ and lower colouring $y_{1}, \dots, y_{k}$, then the attached morphism $\varphi_{b}$ belongs to $\Mor_{\mathfrak{A}}(x_{1}\otimes \dots\otimes x_{n}, y_{1}\otimes \dots \otimes y_{k})$.
\end{de}

The horizontal concatenation of partitions still makes sense in this context, as well as the vertical concatenation if we add the rule that morphisms of blocks are composed in the process. To define the $*$-operation or the rotation, we need some extra structure.

\begin{de}
Let $\mathfrak{A}$ be a C*-tensor category with duals and let $p$ be an $\mathfrak{A}$-decorated partition. Then, $p^{*}$ is defined by applying the $*$-operation to the underlying coloured partition and taking the adjoint of the block morphisms. The rotated version of $p$ obtained by rotating the leftmost upper point to the left of the lower row is defined by applying the operation to the underlying coloured partition and applying Frobenius reciprocity to the involved block morphism.
\end{de}

\begin{de}
A \emph{category of $\mathfrak{A}$-decorated partitions} is a collection $\CC$ of $\mathfrak{A}$-decorated partitions which is stable under vertical and horizontal concatenation, under the $*$-operation, under rotation and contains the partition $\vert$ with $x$ on both end and the identity operator on the block, for every object $x$ of $\mathfrak{A}$. If $\CC$ is a category of $\mathfrak{A}$-decorated partitions and $\theta\in \C$, the associated partition C*-tensor category $\mathfrak{C}(\CC, \theta)$ is defined in the obvious way.
\end{de}

Le $\mathfrak{A}$ be a C*-tensor category with duals and let $F$ be a unitary fiber functor on $\mathfrak{A}$. Using some twisted tensor products with the maps $T_{p}$ (see \cite[Notation 3.6]{lemeux2014free}), $F$ extends to a unitary fiber functor on $\mathfrak{C}(\CC, N)$ yielding a compact quantum group. One can go back to the beginning of Section \ref{sec:general} and try to develop a theory of "decorated partition quantum groups". The results of \cite{freslon2013representation} may have a generalization to this setting, even though the presence of morphisms on the blocks make things more complicated. For example, the definition of the through-block decomposition, which is a central combinatorial tool, would have to be modified. We will not go into these considerations now, but simply give examples.

\begin{ex}
Consider a colour set $\A$ and build a C*-tensor category $\mathfrak{A}(\A)$ in the following way : objects of $\mathfrak{A}(\A)$ are all words on $\A$, the tensor product is the concatenation and between any two objects there is exactly one morphism. Then, there is an obvious bijection between categories of $\A$-coloured partitions and categories of $\mathfrak{A}(\A)$-decorated partitions. In other words, the setting of decorated partitions contains the setting of coloured partitions.
\end{ex}

\begin{ex}\label{ex:freewreath}
Let $\G$ be a compact quantum group and consider its category of finite-dimensional representations $\mathfrak{C}_{\G}$. Let $\CC_{\G}$ be the category of all $\mathfrak{C}_{\G}$-decorated noncrossing partitions. Then, for any integer $N\geqslant 4$, the compact quantum group associated to $\CC_{\G}$ is the free wreath product $\G\wr_{\ast}S_{N}^{+}$. This is one of the main results of \cite{lemeux2014free}.
\end{ex}

This example suggests an extension of the setting of fusion sets and free fusion semirings. Let us assume that the fusion operation does not take values in $S\cup\{\emptyset\}$ but in $\N[S]$. The definition of the associated fusion semiring is clear. As a typical example, consider the set $S$ of isomorphism classes of irreducible objects in a C*-tensor category with duals and set
\begin{equation*}
x\ast y = \sum_{z\subset x\otimes y}m_{z}^{x\otimes y}z,
\end{equation*}
where $m_{z}^{x\otimes y}$ denotes the multiplicity of the inclusion. If this category is the representation category of a compact quantum group, then we get the fusion semiring of $\G\wr_{\ast}S_{N}^{+}$ by \cite{lemeux2014free}.

\begin{ex}\label{ex:classicalwreath}
Let $\mathfrak{A}$ be any C*-tensor category and consider the category whose objects are finite families of objects of $\mathfrak{A}$ and morphisms are given by all $\mathfrak{A}$-decorated partitions. According to the work of M. Mori \cite{mori2012representation}, this is the wreath product of the category $\mathfrak{A}$ by $S_{N}$. In particular, when $\mathfrak{A}$ is the category of modules over an algebra $B$, we obtain the category of modules over the algebra $B\wr S_{N}$. When $\mathfrak{A}$ is the category of representations of a compact group $G$, this yields the category of representations of $G\wr S_{N}$. If $G$ is abelian, this is the same as the partition quantum group associated to the category of all $\h{G}$-coloured partitions such that in each block, the product of elements on the upper row is equal to the product of elements on the lower row (this is the only condition under which a morphism exists between these representations). This proves Proposition \ref{prop:abelianizationwreath}.
\end{ex}

The construction of M. Mori \cite{mori2012representation} was inspired by the aforementioned work of P. Deligne \cite{deligne2007categorie} and its extensions by F. Knop \cite{knop2006construction, knop2007tensor}. It is in fact even more general than the description of the example : for every $\theta\in \C$, one can build a $2$-functor $S_{\theta}$ from the $2$-category of C*-tensor categories to the $2$-category of pseudo-abelian categories which "interpolates" the "wreath-product-by-$S_{N}$" $2$-functor. The construction straightforwardly generalizes, so that to any category of uncoloured partitions $\CC$ is associated a $2$-functor $F_{\theta}^{\CC}$. Because of example \ref{ex:freewreath}, it is reasonable to say that for $\CC = NC$, this functor interpolates free wreath products by $S_{N}^{+}$. It is not known to us whether a general definition of permutational wreath product of a C*-tensor category by a quantum group exists. It would be interesting to compare such a categorical construction with this partition approach. We also believe that the decorated approach can fill the gap between our work and that of T. Banica and A. Skalski on quantum isometry groups alluded to before. In fact, a consequence of the work of M. Mori \cite{mori2012representation} is that all the examples of Subsection \ref{subsec:classical} arise from categories of decorated partitions. Since the constructions given in \cite{banica2012quantum} for quantum isometry groups are very similar, one may hope that they can also be described using decorated partitions.

\appendix
\renewcommand{\thesection}{A}
\setcounter{section}{0}

\section{Simplicity and trace uniqueness for some free quantum groups (by A. Chirvasitu)}

\subsection{Free quantum groups}

The goal of this subsection is to prove the following result :

\begin{thm}\label{th.main}
Let $\G$ be a free compact quantum group whose fusion semiring $R^{+}(\G)$ has a fusion set $S$ with at least two elements and satisfying the following property : 
\begin{equation}\label{eq.main}
\text{\emph{There is no element} }\tau\in S\text{ \emph{such that} }\forall x\in S, \\overline{x}*x = \tau. 
\end{equation}
Then, the reduced C*-algebra $C_{\text{red}}(\G)$ of $\G$ is simple and its Haar state $h$ is the only one satisfying a KMS condition of the form 
\begin{equation}\label{eq.1}
 \psi(xy) = \psi(y (f_{s}*x*f_{t})),\ \forall x,y\in C_{\text{red}}(\G)
\end{equation}
for some $s, t\in\R$, where $(f_{z})_{z}$ is the family of Woronowicz characters (see  \cite[Thm 5.6]{woronowicz1987compact}).
\end{thm}

\begin{rem}
According to the classification of admissible fusion sets given above, the assumption of Theorem \ref{th.main} means that we exclude the cases $S = \OO$ and $S = \Gamma$. However, the conclusion is known to hold also in most of these cases by \cite{vaes2007boundary}, \cite{brannan2012reduced} and \cite{lemeux2013haagerup}. 
\end{rem}

The first result of this form was proved by T. Banica in \cite{banica1997groupe} for the free unitary quantum groups $U_{N}^{+}$ and a close analogue of Theorem \ref{th.main} was conjectured in \cite{banica2009fusion}. Finally, \ref{th.main} answers the conjecture at the very end of \cite{freslon2013representation} for fusion sets satisfying \eqref{eq.main}. The proof is a tweak on that of \cite[Thm 3]{banica1997groupe}, so we will indicate the modifications that need to be made to that proof in order to get Theorem \ref{th.main}. Whenever possible we replicate T. Banica's notation so as to make it easier for the reader to follow the slightly modified argument. Let $\beta\ne \gamma\in S$ be two different elements. We define the following objects :
\begin{itemize}
\item $D\subset F(S)$ is the set of words on $S$ whose first letter is \emph{not} $\beta$ (excluding the empty word).
\item $E\subset F(S)$ is the set of words on $S$ that are either empty or start with $\beta$.
\item $F\subset F(S)$ is the set of words of length at least two whose first and last letters are $\beta$ and $\overline{\beta}$ respectively.
\item $r_{1}$, $r_{2}$ and $r_{3}$ are words of length at least four on $S$ starting with $\beta\gamma\beta$, $\beta\gamma^{2}\beta$ and $\beta\gamma^{3}\beta$ respectively and ending in some letter different from $\overline{\beta}$.
\end{itemize}
We make the additive semigroup $\N[F(S)]$ into a semiring via the product $\otimes$ as in the main paper above, and introduce the binary operation $\circ$ on the collection of all subsets of $F(S)$ by defining $X\circ Y$ to be the set of all $z\in F(S)$ which appear in the decomposition of  $x\otimes y$ for some $x\in X$ and $y\in Y$. With all of this in place, we leave it to the reader to show that the conclusion of \cite[Lem 12]{banica1997groupe} can be replicated verbatim :

\begin{lem}\label{lemme12}
The sets $D$ and $E$ partition $F(S)$, $F\circ D\cap D=\emptyset$, and the sets $r_{i}\circ E$ are pairwise disjoint for $i = 1, 2, 3$.
\end{lem}

The proof of \cite[Cor 3]{banica1997groupe} now goes through (with $F$ substituted for $\{\beta\dots\alpha\}$ in the statement). The second modification we have to make is to \cite[Lem 13]{banica1997groupe}, which states that in the case of $U_N^+$, for any finite subset $R\subset F(S)$ the set $\left(\beta\overline{\beta}\right)^{M}\circ R\circ \left(\beta\overline{\beta}\right)^{M}$ is contained in the set $F\cup\{\text{empty word}\}$. This need not be true in general, as the case where $R = \{\text{empty word}\}$ and the fusion $\beta *\overline{\beta}$ is not $\emptyset$ shows. 

\begin{de}
The \emph{support} of an element $x\in \N[F(S)]$ is the set of all
words on $S$ that have non-zero coefficients in $x$. The element $x$ is said to be \emph{supported} in a set if its support is contained
in that set. 
\end{de}

We need the following auxiliary result, whose proof is a simple case chase.

\begin{lem}
Let $\eta\in S$ be an arbitrary element and let $\tau$ be the fusion
$\eta*\overline{\eta}$. Then, for every $x\in\N[F(S)]$ we can find
some exponent $M$ such that the element
\begin{equation*}
\left(\eta\overline{\eta}\right)^{M} \otimes x\otimes \left(\eta\overline{\eta}\right)^{M}
\end{equation*}
is supported in the set of words of one of the forms
\begin{equation}\label{eq.2}
\eta\dots \overline{\eta},\quad \eta\dots \tau,\quad
\tau\dots\overline{\eta},\quad \tau,\quad \text{empty word}. 
\end{equation}
\end{lem}

Now we claim that in the setting of Theorem \ref{th.main}, \cite[Cor 4]{banica1997groupe} goes through with $F$ in place of $\{\beta\dots\alpha\}$ :

\begin{lem}\label{lemma.aux}
Let $x\in C_{\text{red}}(\G)$ be a self-adjoint element annihilated by the Haar state
$h$. Then,
\begin{enumerate}
\renewcommand{\labelenumi}{(\roman{enumi})}
\item There is a linear, unital, $h$-preserving self-map $W$ of $C_{\text{red}}(\G)$ of the form $z\mapsto \sum_{i} b_{i}zb_{i}^{*}$ such that $W(x)$ is supported in $F$.
\item For fixed $s, t\in\R$ there is a positive number $L$ and a linear $h$-preserving self-map $U$ of $C_{\text{red}}(\G)$ which preserves all states $\psi$ on $C_{\text{red}}(\G)$ satisfying \eqref{eq.1} and such that $U(x)$ is supported in $F$.
\end{enumerate}
\end{lem}

\begin{proof}
If $\beta*\overline{\beta} = \emptyset$ then the procedure in the
original proof of T. Banica works verbatim. Otherwise, apply that
procedure first for $r = \left(\eta\overline{\eta}\right)^{M}$, with a choice of $\eta$ such that if the fusion $\beta*\overline{\beta}$ is not $\emptyset$, then it is different from $\tau = \eta*\overline{\eta}$. Assume that $\tau\in S$ (so that it is not $\emptyset$). Then, through that process, we obtain the conclusion of the lemma for some maps $W'$ and $U'$ in all respects except that the supports of $W'(x)$ and $U'(x)$ are contained in \eqref{eq.2} rather than $F$. Moreover, by $h$-preservation and the fact that $h$ annihilates $x$, these supports do not contain the empty word ; in other words, the words in these supports are of the form
\begin{equation}\label{eq.3}
\eta\dots \overline{\eta},\quad \eta\dots \tau,\quad
\tau\dots\overline{\eta},\quad \tau.
\end{equation} 
Now apply the same procedure once more to $W'(x)$ (for part $(i)$) and
$U'(x)$ (in part $(ii)$) with $r$ being some appropriately large power of $\beta\overline{\beta}$ instead. The condition $\beta*\overline{\beta}\ne \tau$ implies that neither $\tau*\beta=\beta$ nor $\overline{\beta}*\tau=\overline{\beta}$ hold. Moreover, $\beta\ne\eta$ and $\beta,\overline{\beta}\ne\tau$. The above inequalities imply that for large enough $P$, $\left(\beta\overline{\beta}\right)^{P}\circ W'(x)\circ \left(\beta\overline{\beta}\right)^{P}$ is supported in $F$. To illustrate how this works, let us just see what happens if $W'(x)$ is supported in $\{\tau\}$ alone (the other cases from \eqref{eq.2} are even easier). Because $\beta\ne\tau$, the product $(\left(\beta\overline{\beta}\right))^{P}\otimes \tau$ is
\begin{equation*}
\beta\overline{\beta}\dots\beta \overline{\beta}\tau + \beta\overline{\beta}\dots\beta \left(\overline{\beta}*\tau\right).
\end{equation*} 
Now, because $\beta*\overline{\beta}\ne\tau$, the rightmost letter in
this expression is not $\overline{\beta}$, and hence multiplying on
the right with $\left(\left(\beta\overline{\beta}\right)\right)^{P}$ will not collapse the second summand too much. Similarly, the first summand will not collapse because $\overline{\beta}\ne\tau$. Finally, if $\eta*\overline{\eta}=\emptyset$ then the argument above
goes through even faster : this time, instead of \eqref{eq.2}, the
supports of $W'(x)$ and $U'(x)$ will be contained in
$\{\eta\dots\overline{\eta}\}$ alone.
\end{proof}

\begin{proof}[Proof of Theorem \ref{th.main}]
Replicate the proof of \cite[Thm 3]{banica1997groupe} substituting
Lemma \ref{lemma.aux} for \cite[Cor 4]{banica1997groupe}.
\end{proof}

\subsection{Free products}

The result of this subsection is :

\begin{thm}\label{th.main_bis}
If a compact quantum group $\G$ breaks up non-trivially as a free product $\G_{1}\ast \G_{2}$ and either $\G_{1}$ or $\G_{2}$ has at least two non-trivial representations then the conclusion of Theorem \ref{th.main} holds.
\end{thm}

In general, for a quantum group $\mathbb{H}$ we denote its set of irreducible representations by $\Irr(\mathbb{H})$. Recall (e.g. \cite[Thm 3.10]{wang1995free}) that the irreducible representations of $\G$ are of the form $s_{1}\otimes\dots\otimes s_{k}$, where the $s_{j}$'s are alternately non-trivial elements in $\Irr(\G_{1})$ and $\Irr(\G_{2})$ ; the empty tensor product corresponds to the trivial representation. For this reason we will suppress the tensor product symbols and refer to the elements of $\Irr(\G)$ as words, with the letters being elements of $\mathcal{B}_{1}=\Irr(\G_{1})-\{\text{trivial representation}\}$ and $\mathcal{B}_{2} = \Irr(\G_{2})-\{\text{trivial representation}\}$. We define supports for elements in $R^{+}(\G)$ as subsets of $\Irr(\G)$ in the obvious way, as in the previous section. We now follow the same plan as before ; in fact, this time the arguments will be simpler. Suppose $|\mathcal{B}_{2}|\geqslant 2$, and fix $\beta\in\mathcal{B}_{1}$ and $\gamma\neq \xi\in\mathcal{B}_{2}$. Define
\begin{itemize}
\item $D\subset \Irr(\G)$ to be the set of words whose first letter is \emph{not} in $\mathcal{B}_{1}$ (excluding the empty word).
\item $E\subset \Irr(\G)$ to be the set of words that are either empty or start with a letter in $\mathcal{B}_{1}$.
\item $F\subset \Irr(\G)$ to be the set of words of length at least two whose first and last letters are $\beta$ and $\overline{\beta}$ respectively.
\item $r_{1}$, $r_{2}$ and $r_{3}$ to be words starting with $\beta\gamma\beta\gamma\beta$, $\beta\gamma\beta\xi\beta$ and $\beta\xi\beta\xi\beta$ respectively and ending in some letter from $\mathcal{B}_{2}$.
\end{itemize}

This once more makes \cite[Lem 12]{banica1997groupe} valid, as well as \cite[Lem 13]{banica1997groupe} in the form :

\begin{lem}
For any finite subset $R\subset \Irr(\G)$ the set $(\beta\gamma)^{M}\circ R\circ\left(\overline{\beta\gamma}\right)^{M}$ is supported in $F\cup\{\text{empty word}\}$ for sufficiently large $M$.
\end{lem}

\begin{proof}[Proof of Theorem \ref{th.main_bis}]
Substitute $F$ for $\{\beta\dots\alpha\}$ in the statements of \cite[Cor 3]{banica1997groupe} and \cite[Cor 4]{banica1997groupe} and set $r = (\beta\gamma)^{M}$ in the proof of the latter result ; the proof of \cite[Thm 3]{banica1997groupe} once more goes through. 
\end{proof}

\bibliographystyle{amsplain}
\bibliography{../../../quantum}

\end{document}